\documentclass[11pt]{amsart}

\usepackage{epsfig,overpic}
\usepackage[usenames,dvipsnames,svgnames,table]{xcolor}
\usepackage[hyphens]{url}
\usepackage[pagebackref,linktocpage=true,colorlinks=true,linkcolor=Blue,citecolor=BrickRed,urlcolor=RoyalBlue]{hyperref}
\usepackage[msc-links,abbrev]{amsrefs}
\usepackage{amsmath,amsthm,amssymb}
\usepackage{comment}

\newcommand{\R}{\mathbf{R}}
\newcommand{\M}{\mathbf{M}}
\newcommand{\K}{\mathcal{K}}
\newcommand{\MA}{\mathcal{M}}
\newcommand{\C}{\mathcal{C}}
\renewcommand{\L}{\mathcal{L}}

\newcommand{\be}{\begin{equation}}
\newcommand{\ee}{\end{equation}}

\newcommand{\ol}{\overline}
\newcommand{\CAT}{\operatorname{CAT}}
\newcommand{\length}{\operatorname{length}}
\newcommand{\cs}{\operatorname{cs}}
\newcommand{\sn}{\operatorname{sn}}
\newcommand{\conv}{\operatorname{conv}}
\newcommand{\signdist}{\operatorname{sgndist}}
\newcommand{\inte}{\operatorname{int}}

\renewcommand{\S}{\mathbf{S}}
\renewcommand{\tilde}{\widetilde}
\renewcommand{\angle}{\measuredangle}
\renewcommand{\epsilon}{\varepsilon}

\theoremstyle{plain}
\newtheorem{theorem}{Theorem}[section]
\newtheorem{lemma}[theorem]{Lemma}
\newtheorem{proposition}[theorem]{Proposition}

\theoremstyle{definition}

\newtheorem{note}[theorem]{Note}

\theoremstyle{remark}

\begin{document}

\title[A quantitative Schur comparison theorem]{A quantitative Schur comparison theorem\\for curves in CAT($\mathbf{k}$) spaces}

\author{Mohammad Ghomi}
\address{School of Mathematics, Georgia Institute of Technology,
Atlanta, GA 30332}
\email{ghomi@math.gatech.edu}
\urladdr{www.math.gatech.edu/~ghomi}

\author{John Ioannis Stavroulakis}
\address{School of Mathematics, Georgia Institute of Technology,
Atlanta, GA 30332}
\email{jstavroulakis3@gatech.edu}

\begin{abstract}
We obtain a quantitative form of Schur's comparison theorem for curves with finite total curvature in
$\CAT(k)$ spaces. This sharpens and extends the classical arm and bow lemmas in Euclidean
space, as well as their Riemannian analogues. The proof is based on a comparison formula for curves in model planes, expressed in terms of curvature measures and a notion of moment arm borrowed from mechanics. Another ingredient is a refinement of Reshetnyak's theorem that controls the curvature of the majorizing curve.
 \end{abstract}

\date{\today\,(Last Typeset)}



\maketitle

\section{Introduction}
Schur's classical  theorem \cite{schur1921,hopf1989,chern1967} asserts that decreasing the curvature of a convex arc  moves its endpoints farther apart. Here we establish a sharp quantitative form of this principle for curves in spaces of curvature bounded above. To state our results precisely, 
let $X_k$ be a $\CAT(k)$ space,  and $\alpha\colon [a,b]\to  X_k$ be  a continuous map or \emph{curve}. The \emph{total curvature} \cite{sullivan2008,milnor1950} of $\alpha$ on any interval $I:=[c,d]\subset[a,b]$ is given by
$$\kappa_\alpha(I):=\sup_P\sum_{i=1}^{N-1}\big(\pi-\angle_\alpha(t_i)\big),$$
where the supremum ranges over all partitions $P=\{c=t_0<\cdots<t_N=d\}$  of $I$ with $\alpha(t_i)\neq\alpha(t_{i+1})$, and
$0\leq\angle_\alpha(t_i)\leq \pi$ is the Alexandrov angle \cite{bbi2001,bridson-haefliger1999} between the geodesic segments 
$\alpha(t_i)\alpha(t_{i-1})$ and $\alpha(t_i)\alpha(t_{i+1})$.  If $ X_k$ is a  Riemannian manifold and $\alpha$ is $\C^2$, then
$
 \kappa_\alpha
$
 is the integral of the norm of geodesic curvature \cite{lopez-mateos-masque2010}. We say  $\alpha$ is  \emph{convex} if it is injective on $(a,b)$, and  $\alpha([a,b])$ lies on the boundary of its convex hull, denoted by $\conv(\alpha)$. 
 For curves $\alpha$, $\beta\colon[a,b]\to  X_k$ we write $\kappa_\beta\leq\kappa_\alpha$ if  $\kappa_\beta(I) \leq \kappa_\alpha(I)$ for all intervals $I\subset[a,b]$.  Let $\M^2_k$ denote the model plane of constant curvature $k$. 

\begin{theorem}\label{thm:main}
Let $\alpha \colon [0,\ell] \to \M^2_k$ and  $\beta \colon [0,\ell] \to  X_k$ be unit speed curves. Suppose that $\alpha$ is convex, $\kappa_\beta \leq \kappa_\alpha$,
and $\ell<\pi/\sqrt{k}$ if $k>0$.
Then
\begin{equation}\label{eq:schur}
\bigl|\beta(0)\beta(\ell)\bigr|\geq \bigl|\alpha(0)\alpha(\ell)\bigr|.\end{equation}
Equality holds only if the map $\alpha(t)\mapsto\beta(t)$ extends to an isometry 
between $\conv(\alpha)$ and $\conv(\beta)$.
\end{theorem}

We will refine this result below (Theorem \ref{thm:main2}) by computing the deficit in \eqref{eq:schur} through a virtual work formula when one curve is deformed into another by redistribution of curvature. This involves the notion of moment arm from classical mechanics, which we adapt to curves in model planes, by thinking of a convex arc as a generalized hinge whose endpoints are pulled apart by a force parallel to the endpoint geodesic segment.

Our findings improve several earlier results with regard to regularity, ambient space, sharpness, or rigidity. The case where $\alpha$ is polygonal, known as  the ``arm lemma'' \cite{akp2024,aigner-ziegler1999,sabitov2004,schoenberg-zaremba1967}, goes back to Cauchy and Legendre.  For piecewise $\C^2$ curves in $\R^n$, Theorem \ref{thm:main} originates with Schur \cite{schur1921}, and has been called the ``bow lemma'' \cite{petrunin-zamora2022,gromov2025}. For piecewise $\C^2$ curves in  $\S^2$, it was established recently by Ni \cite{ni2023}.
For  $\C^1$ curves in Cartan-Hadamard manifolds $M$, the endpoint inequality \eqref{eq:schur} was proved in \cite{ghomi2026-total}; in the case where  $M$ has curvature at most $k\leq 0$ and $\alpha$ is $\C^2$, it was proved in \cite{ghomi2025-convexity}. See Gromov \cite[Sec. 8 \& 9]{gromov2025} and \cite{mendes2025,howard2024,lopez2011,petrunin2025} for other recent variations and applications.

To prove Theorem~\ref{thm:main}, we use a refinement of Reshetnyak's majorization theorem to replace $\beta$ by a convex curve $\tilde\beta$ in $\M^2_k$ with the same length, endpoint distance, and with curvature
$\kappa_{\tilde\beta}\leq\kappa_\beta\leq\kappa_\alpha$. Thus the problem is reduced to  the model plane, where our moment arm identity, applied through a variational argument, yields
$$
|\alpha(0)\alpha(\ell)|
\leq
|\tilde\beta(0)\tilde\beta(\ell)|
=
|\beta(0)\beta(\ell)|.
$$
The characterization of the equality case is obtained by tracing the equality conditions in the moment arm identity and in Reshetnyak's majorization. The full argument is presented as follows. Section \ref{sec:curvature-measures} reviews curvature measures for curves of finite total curvature in Riemannian surfaces and proves the basic existence and uniqueness results that we need. Section \ref{sec:moment-arm} introduces the moment arm and establishes the key identity (Theorem \ref{thm:moment-id}). In Section \ref{sec:model-plane} we apply this identity to prove Theorem \ref{thm:main} for model planes. Section \ref{sec:majorization} refines Reshetnyak's theorem by choosing the majorizing curve with no greater curvature (Theorem \ref{thm:reshetnyak-curvature}). Finally, Section \ref{sec:main} combines these ingredients to establish our main result (Theorem \ref{thm:main2}) which encompasses Theorem \ref{thm:main}. Appendix \ref{sec:polygonal} gives a quicker proof of \eqref{eq:schur} without rigidity.

\section{Curvature Measures}\label{sec:curvature-measures}
We assume that all curves $\alpha\colon[a,b]\to X_k$ have
\emph{finite total curvature}, i.e., $\kappa_\alpha([a,b])<\infty$. In particular $\alpha$ is rectifiable \cite{alexandrov-reshetnyak1989, maneesawarng-lenbury2003} and so admits arclength parametrization, in which case we say it is
\emph{unit speed}. We will further assume that $\length(\alpha)<\pi/\sqrt{k}$
when $k>0$, which ensures the uniqueness of geodesic segments connecting points of $\alpha$. 

The study of curves with finite total curvature goes back to Milnor \cite{milnor1950}. Their fundamental properties were established in $\R^n$ by Alexandrov-Reshetnyak \cite{alexandrov-reshetnyak1989,reshetnyak2005,sullivan2008}, and extended to $\CAT(k)$  spaces \cite{alexander-bishop1998,alexander-bishop1996,maneesawarng-lenbury2003,karuwannapatana-maneesawarng2007} and Riemannian manifolds \cite{mucci-saracco2021,mucci-saracco2024,lopez-mateos-masque2010} by others. We use these results to develop the basic curvature calculus that we need.

Let $M^2$ be a complete simply connected  Riemannian surface. We assume that $M^2$ is oriented by a consistent choice of (counterclockwise) rotation $J\colon T_pM\to T_p M$ by $\pi/2$ in each tangent space. Let
$\alpha\colon[0,\ell]\to M^2$ be a unit speed curve of finite total curvature. Then $\alpha$ is \emph{one-sidedly smooth} in the sense of Alexandrov-Reshetnyak \cite{alexandrov-reshetnyak1989,mucci-saracco2021}. This means that, in local
coordinates, the left and right derivatives $\alpha_{\pm}'$ exist  everywhere and are one-sidedly continuous, i.e.,
$
\alpha'_-(t)=\lim_{s\to t^-}\alpha'_-(s)=\lim_{s\to t^-}\alpha'_+(s),
$
and
$
\alpha'_+(t)=\lim_{s\to t^+}\alpha'_-(s)=\lim_{s\to t^+}\alpha'_+(s),
$
for every $t\in(0,\ell)$, with the obvious modifications at the endpoints. We set
$
T_\pm(t):=\alpha'_{\pm}(t),
$
for $t\in(0,\ell)$, and 
$
T_+(0):=\alpha'_+(0)
$,
$
T_-(\ell):=\alpha'_-(\ell).
$
Then the \emph{angle} and \emph{exterior angle} of $\alpha$ at $t\in(0,\ell)$ are given, respectively, by
$$
\angle_\alpha(t):=\angle\big(-T_+(t),T_-(t)\big)\in[0,\pi],\qquad\qquad \tilde\angle_\alpha(t):=\pi-\angle_\alpha(t).
$$
Furthermore we define the \emph{unit tangent vector field} of $\alpha$ by
$$
T(t):=T_-(t)\quad\text{for }0<t\leq\ell,\qquad
T(0):=T_+(0).
$$
Since $\alpha$ is one-sidedly smooth, $T$ is left-continuous.
A point $t\in(0,\ell)$ is  \emph{smooth} if $\angle_\alpha(t)=\pi$; otherwise it is a \emph{corner}, and is called a
\emph{cusp} if  $\angle_\alpha(t)=0$.
Choose a smooth orthonormal frame $E_1$, $ E_2$, with $J(E_1)=E_2$, on an open set containing
$\alpha([0,\ell])$. Let
$$
\omega(v):=\big\langle\nabla_vE_1,E_2\big\rangle
$$
be the corresponding connection form. Since $\alpha$ has finite total
curvature, $T$ has bounded variation, i.e., $\langle T,E_i\rangle$
are BV  functions \cite[p.~144]{sullivan2008} \cite[p. 531]{mucci-saracco2021}. By the
left-continuity of $T$, the angle  which $T$ makes with $E_1$ may be constructed
to be left-continuous, with each jump from $T_-(t)$ to $T_+(t)$ lying in
$(-\pi,\pi]$ (so, by convention, cusps correspond to rotation by $\pi$). This gives a left-continuous BV function $\theta\colon[0,\ell]\to\R$, called the \emph{turning angle} of
$\alpha$, unique after prescribing $\theta(0)$, such that
$
T=\cos\theta\,E_1+\sin\theta \,E_2.
$
 The  \emph{(cumulative) curvature function} of $\alpha$ is then defined by 
\begin{equation}\label{eq:integrated-frenet}
\K_\alpha(t):=\theta(t)-\theta(0)+\int_0^t\omega\big(T(r)\big)\,dr,
\end{equation}
for $0\leq t\leq\ell$.
Thus $\K_\alpha$ is left-continuous and BV. Consequently, the \emph{(signed) curvature measure} of $\alpha$ may be defined as the Lebesgue-Stieltjes measure $d\K_\alpha$, given by
$$
d\K_\alpha([s,t)):=\K_\alpha(t)-\K_\alpha(s).
$$
Since the connection term $\omega$ is continuous, the
atoms of $d\K_\alpha$ are precisely the jumps of $\theta$ at the corners of $\alpha$. Hence
$
d\K_\alpha(\{t\})\in(-\pi,\pi],
$
and
$
\big|d\K_\alpha(\{t\})\big|=\tilde\angle_\alpha(t).
$
If another
orthonormal frame is obtained by rotation through an angle $\phi$, then the
new turning angle is $\theta-\phi\circ\alpha$, while the new connection form
is $\omega+d\phi$. These changes cancel in
\eqref{eq:integrated-frenet}. Thus $\K_\alpha$ depends only on $\alpha$.
If $\alpha$ is $\C^2$, 
then
$
d\K_\alpha=k_\alpha\,dt,
$
where
$
k_\alpha:=\big\langle\nabla_TT,J(T)\big\rangle
$
 is the geodesic curvature.
For $0\leq s<t\leq\ell$,
$$
\kappa_\alpha([s,t])=|d\K_\alpha|\big((s,t)\big).
$$
This equation may be taken as the definition of the 
\emph{(unsigned) curvature measure} $|d\K_\alpha|$.
 For curves $\alpha$, $\beta\colon[0,\ell]\to M^2$,  
we have $\kappa_\beta\leq\kappa_\alpha$ 
 if and only if $|d\K_\beta|\leq |d\K_\alpha|$. Using \eqref{eq:integrated-frenet}, we now show the continuous dependence
of curves on their curvature functions. We write $|pq|$ for the distance between points $p$ and $q$.

\begin{lemma}\label{lem:continuous}
Let $\alpha,\alpha_m\colon[0,\ell]\to M^2$ be unit speed curves of finite
total curvature with
$
        \alpha_m(0)=\alpha(0),
$
$
        T_m(0)=T(0).
$
Then there is a constant $C$, independent of $m$, such that
\begin{equation}\label{eq:L1-continuity-estimate}
\|T_m-T\|_{L^1([0,\ell])}
+
\sup_{0\leq t\leq\ell}|\alpha_m(t)\alpha(t)|
\leq
C\|\K_{\alpha_m}-\K_\alpha\|_{L^1([0,\ell])}.
\end{equation}
Furthermore, if
$
\|\K_{\alpha_m}-\K_\alpha\|_{L^1([0,\ell])}\to0
$
and
$
\K_{\alpha_m}(\ell)\to \K_\alpha(\ell),
$
then $T_m(\ell)\to T(\ell)$.

\end{lemma}

\begin{proof}
We may assume that $\alpha$, $\alpha_m$ are  contained in a coordinate neighborhood
which carries a smooth oriented orthonormal frame $E_1,E_2$. Let
$\theta$, $\theta_m$ be the corresponding turning angles, chosen so that
$\theta_m(0)=\theta(0)$, and set
$$
V(p,\theta):=\cos\theta E_1(p)+\sin\theta E_2(p).
$$
In coordinates,
\begin{equation}\label{eq:alpham}
\alpha_m(t)-\alpha(t)
=
\int_0^t
\bigl(
V(\alpha_m(s),\theta_m(s))-V(\alpha(s),\theta(s))
\bigr)\,ds.
\end{equation}
By \eqref{eq:integrated-frenet} we have
\begin{equation}\label{eq:thetam}
\theta_m(t)-\theta(t)
=
\K_{\alpha_m}(t)-\K_\alpha(t)
-
\int_0^t
\bigl(
\omega(T_m(s))
-
\omega(T(s))
\bigr)\,ds .
\end{equation}
The maps
$
        (p,\theta)\mapsto V(p,\theta)
$
and
$
        (p,\theta)\mapsto\omega(V(p,\theta))
$
are Lipschitz. Hence there is a constant $C$,
independent of $m$, such that
\begin{equation}\label{eq:lipschitz}
|T_m-T|
+
|\omega(T_m)-\omega(T)|
\le
C\bigl(|\alpha_m-\alpha|+|\theta_m-\theta|\bigr).
\end{equation}
Next we define
$$
F_m(t):=\int_0^t
\bigl(
|\alpha_m(s)-\alpha(s)|+|\theta_m(s)-\theta(s)|
\bigr)\,ds.
$$
From \eqref{eq:alpham}, \eqref{eq:thetam}, and \eqref{eq:lipschitz}, we get, for every $t$,
$$
|\theta_m(t)-\theta(t)|
\le
|\K_{\alpha_m}(t)-\K_\alpha(t)|
+
CF_m(t),
\qquad
|\alpha_m(t)-\alpha(t)|
\le
CF_m(t).
$$
Thus, after increasing $C$,
$$
F_m'(t)
\le
|\K_{\alpha_m}(t)-\K_\alpha(t)|
+
CF_m(t)
$$
for almost every $t$. By Gr\"{o}nwall's inequality,
$
        F_m(\ell)
        \le
        e^{C\ell}
        \|\K_{\alpha_m}-\K_\alpha\|_{L^1([0,\ell])}.
$
Finally, by \eqref{eq:lipschitz}, the definition of $F_m$, and the preceding
estimate,
$$
\|T_m-T\|_{L^1([0,\ell])}
\leq
C\|\K_{\alpha_m}-\K_\alpha\|_{L^1([0,\ell])}.
$$
Also, by \eqref{eq:alpham},
$
\sup_{0\le t\le\ell}|\alpha_m(t)-\alpha(t)|
\leq
C\int_0^\ell
\bigl(
|\alpha_m-\alpha|+|\theta_m-\theta|
\bigr).
$
Combining the last two estimates  gives
\eqref{eq:L1-continuity-estimate}. The final assertion follows by evaluating
\eqref{eq:thetam} at $t=\ell$. Indeed, by hypothesis, \eqref{eq:lipschitz},
and \eqref{eq:L1-continuity-estimate}, $\theta_m(\ell)-\theta(\ell)\to0$.
Together with $\alpha_m(\ell)\to\alpha(\ell)$, this gives
$T_m(\ell)\to T(\ell)$.
\end{proof}

Using  Lemma~\ref{lem:continuous}, we obtain the following existence and uniqueness result which extends \cite[Thm.~5.10.5]{alexandrov-reshetnyak1989} for curves in $\R^2$ to Riemannian surfaces.

\begin{proposition}\label{prop:prescribed-curvature}
Let $\mu$ be a finite signed Borel measure on $[0,\ell]$ with $\mu(\{0\})=0=\mu(\{\ell\})$, and 
$
-\pi<\mu(\{t\})\leq \pi,
$
for all $t\in[0,\ell]$.
 Given a point $p\in M^2$ and  unit vector
$T_0\in T_pM^2$, there is a unique unit speed curve
$\alpha\colon[0,\ell]\to M^2$ with
$\alpha(0)=p$, $T(0)=T_0$, and $d\K_\alpha=\mu$.
\end{proposition}

\begin{proof}
First suppose that $\mu=f(t)\,dt$, where $f$ is smooth. Then $\alpha$ is
obtained uniquely by solving
\begin{equation}\label{eq:alpha-prime}
\alpha'=\cos\theta\,E_1(\alpha)+\sin\theta\,E_2(\alpha),
\qquad
\theta'=f-\omega(\alpha'),
\end{equation}
with the prescribed initial data. For the general case, let
$
        F(t):=\mu([0,t))
$
be the left-continuous cumulative function of $\mu$. Extend $F$ to $\R$ by
setting $F(t)=0$ for $t\leq0$ and $F(t)=F(\ell)$ for $t\geq\ell$. Let
$G_m:=\rho_m*F$ be the convolution of $F$ with the standard mollifier
$\rho_m(t):=m\rho(mt)$, where $\rho$ is a nonnegative smooth function
supported in $(-1,1)$ with $\int_\R\rho=1$. Then $G_m$ is smooth. Put
$f_m:=G_m'$ and
$$
F_m(t):=\int_0^t f_m(s)\,ds=G_m(t)-G_m(0).
$$
Then $F_m\to F$ in $L^1([0,\ell])$. Let $\alpha_m$ be the solution to \eqref{eq:alpha-prime} with $f=f_m$ and initial data
$\alpha_m(0)=p$, $T_m(0)=T_0$. Then $\K_{\alpha_m}=F_m$. Since
$F_m$ is Cauchy in $L^1([0,\ell])$, Lemma~\ref{lem:continuous}, applied to
pairs $\alpha_m,\alpha_n$, shows that $\alpha_m$ is Cauchy uniformly on
$[0,\ell]$, while $T_m$ is Cauchy in $L^1([0,\ell])$. Thus $\alpha_m$
converges uniformly to a curve $\alpha$, and $T_m$ converges in $L^1$ to a
field $T$. Then $\alpha'=T$ almost everywhere.

It remains to identify the curvature measure of $\alpha$. Since $\mu$ has no
atoms at $0$ or $\ell$, the extension of $F$ is continuous at these endpoints.
By the standard convergence property of mollifiers, $G_m(t)\to F(t)$ at every
continuity point of $F$. Since $G_m(0)\to F(0)=0$, it follows that
$F_m(t)\to F(t)$ at every continuity point of $F$. Define
$$
\theta(t):=\theta(0)+F(t)-\int_0^t\omega\big(T(s)\big)\,ds .
$$
Then $\theta$ is left-continuous and BV, and
$\theta_m(t)\to\theta(t)$ almost everywhere. Since $T_m\to T$ in $L^1$, after
passing to a subsequence we have $T_m(t)\to T(t)$ for almost every $t$. Passing
to the limit in the identities
$$
T_m(t)=\cos\theta_m(t)E_1\big(\alpha_m(t)\big)+\sin\theta_m(t)E_2\big(\alpha_m(t)\big)
$$
shows that
$
T(t)=\cos\theta(t)E_1(\alpha(t))+\sin\theta(t)E_2(\alpha(t))
$
for almost every $t$. In particular $|T(t)|=1$ for almost every $t$, and hence $\alpha$ is unit speed. Furthermore, $\K_\alpha=F$, and consequently
$d\K_\alpha=dF=\mu$. 

Uniqueness follows from Lemma~\ref{lem:continuous}. Indeed,  the  measure $\mu$ determines the curvature function $\K_\alpha$, and
therefore \eqref{eq:L1-continuity-estimate} forces two curves with the same
initial conditions and curvature function to be identical.
\end{proof}

\section{The Moment Arm Identity}\label{sec:moment-arm}
A \emph{model plane} $\M^2_k$ is a complete simply connected Riemannian surface with constant curvature $k$. Thus, up to a rescaling, $\M^2_k$ is either the Euclidean space $\R^2$, the sphere $\S^2$, or the hyperbolic space $\mathbf{H}^2$. 
Here we introduce the notion of moment arm for  curves in $\M^2_k$, and use it to obtain a comparison formula for their endpoint distance. Let
$$
\sn_k(x):=
\begin{cases}
\sin(\sqrt{k}x)/\sqrt{k}, & k>0,\\
x, & k=0,\\
\sinh(\sqrt{-k}x)/\sqrt{-k}, & k<0
\end{cases}
$$
be the generalized sine function. Let $p,q\in\M^2_k$ be distinct points with
$|pq|<\pi/\sqrt{k}$ when $k>0$. We denote the complete geodesic or \emph{line} through them, oriented from $p$ to $q$, by
$\L_{pq}$. The \emph{left side} of $\L_{pq}$ is the side into which $J(T)$ points, where $T$ is a unit tangent vector in the direction of $\L_{pq}$, and the other side is the \emph{right side}. For any point $o\in\M^2_k$, let
$\signdist(o,\L_{pq})$ be the \emph{signed distance} from $o$ to $\L_{pq}$, positive when $o$ lies in the interior of the right side of $\L_{pq}$,
negative when $o$ lies in the interior of the left side, and zero on
$\L_{pq}$. Set
$$
m(poq):=\sn_k\bigl(\signdist(o,\L_{pq})\bigr).
$$

We call $m(poq)$ the \emph{moment arm} of the \emph{hinge} $poq$, i.e., the union of the oriented geodesic segments $po$ and $oq$.  Thus $m(poq)>0$ when $poq$ turns left; see Figure \ref{fig:hinge}. More generally we say that a convex curve $\alpha\colon[0,\ell]\to \M^2_k$, with $\alpha(0)\neq\alpha(\ell)$, \emph{turns left} 
if it lies to the right of $\L_{\alpha(0)\alpha(\ell)}$. 
 \begin{figure}[h]
\begin{overpic}[height=0.75in]{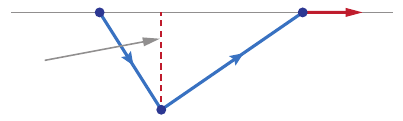}
\put(23,30.5){\small$p$}
\put(74,30.5){\small$q$}
\put(38.5,-2){\small$o$}
\put(0,10){\small$m(poq)$}
\put(-8,27){\small$\L_{pq}$}
\end{overpic}
\caption{}\label{fig:hinge}
\end{figure}
Equivalently the turning angles of polygonal curves inscribed in $\alpha$ are nonnegative, and $d\K_\alpha\geq 0$; these facts are standard in $\R^2$ \cite[Sec. 5.10 \& 7.2.1]{alexandrov-reshetnyak1989} and extend to $\M^2_k$ by projective transformations.
For $k=0$, $m(poq)$ is  the signed distance from $o$ to the line $\L_{pq}$, which is the usual notion of moment
arm in mechanics corresponding to a unit force applied to an endpoint of a hinge along the direction of the endpoint line. 

 Let $\rho\colon[0,\ell]\to\R$ be a left-continuous BV function, with $\rho(0)=0$, and $d\rho$ be the
Lebesgue-Stieltjes measure associated to $\rho$. Suppose also that $d\rho$
has no atoms at $0$ or $\ell$, i.e.,
$d\rho(\{0\})=d\rho(\{\ell\})=0$. Then, for any unit speed curve $\alpha\colon[0,\ell]\to\M^2_k$ with
$\alpha(0)\neq\alpha(\ell)$, we define the \emph{moment arm}  with respect to $\rho$ by
$$
\MA(\alpha,\rho):=\int_0^\ell
m\big(\alpha(0)\alpha(t)\alpha(\ell)\big)\,d\rho(t).
$$
If $\alpha(0)=\alpha(\ell)$, we set $\MA(\alpha,\rho)=0$. In particular, note that if $\alpha$ is a convex curve which turns left and $d\rho\geq 0$, then $\MA(\alpha,\rho)\geq 0$.

\begin{theorem}\label{thm:moment-id}
Let $\alpha,\beta\colon[0,\ell]\to\M^2_k$ be unit speed curves of finite
total curvature. For $0\leq r\leq1$, let
$\alpha_r\colon[0,\ell]\to\M^2_k$ be  unit speed curves  with curvature measure
$
d\K_{\alpha_r}:=d\K_\alpha-rd(\K_\alpha-\K_\beta).
$
Then
\begin{equation}\label{eq:moment-id}
|\beta(0)\beta(\ell)|-|\alpha(0)\alpha(\ell)|
=
\int_0^1\MA\bigl(\alpha_r,\K_\alpha-\K_\beta\bigr)\,dr .
\end{equation}
\end{theorem}

The curves $\alpha_r$ here are well defined by
Proposition~\ref{prop:prescribed-curvature}. Indeed, the endpoint atom condition
is inherited from $d\K_\alpha$ and $d\K_\beta$, while for every $t\in(0,\ell)$,
$$
d\K_{\alpha_r}(\{t\})
=
(1-r)d\K_\alpha(\{t\})+r d\K_\beta(\{t\})\in(-\pi,\pi].
$$ 
In physical terms, equation \eqref{eq:moment-id}, which we call the \emph{moment arm identity}, states that the gain in endpoint distance
is the total virtual work obtained by changing the curvature by
$d(\K_\alpha-\K_\beta)$
under a unit force separating the endpoints. The proof of Theorem \ref{thm:moment-id} is based on the following infinitesimal computation. Note that, as shown in Figure \ref{fig:hinge}, when the hinge turns left, or lies on the right  side of $\L_{pq}$, then the rotation of  $q$ about $o$ increases the distance between $p$ and $q$, provided that it is in the clockwise direction, which is why the angle $\theta$ here appears with a negative sign.

\begin{lemma}\label{lem:hinge-derivative}
Let $poq$ be a hinge in $\M^2_k$, and $q_\theta$ be obtained from $q$ by
rotation about $o$ by the angle $-\theta$. Then
$$
\left.\frac{d}{d\theta}\right|_{\theta=0}|pq_\theta|=m(poq).
$$
\end{lemma}

\begin{proof}
Let $\cs_k:=\sn_k'$ be the generalized cosine function, and let
$\psi:=\angle(poq)$. Suppose first that $o$ lies to the right of the oriented
line $\L_{pq}$. Since $q_\theta$ is obtained by rotating $q$ about $o$ through
the angle $-\theta$, the angle $\angle(poq_\theta)$ is $\psi+\theta$. Thus, for
$k\ne0$, the law of cosines gives
$$
\cs_k(|pq_\theta|)=\cs_k(|po|)\cs_k(|oq|)
+k\,\sn_k(|po|)\sn_k(|oq|)\cos(\psi+\theta).
$$
Differentiating at $\theta=0$, and using $\cs_k'=-k\sn_k$, yields
$$
\left.\frac{d}{d\theta}\right|_{\theta=0}|pq_\theta|
=
\frac{\sn_k(|po|)\sn_k(|oq|)\sin\psi}{\sn_k(|pq|)}
=
\sn_k\bigl(\signdist(o,\L_{pq})\bigr).
$$
If $o$ lies to the left of $\L_{pq}$, then $\angle(poq_\theta)=\psi-\theta$,
and the same computation gives the negative of the preceding quantity. The case $k=0$ follows similarly, using the Euclidean law of cosines.
\end{proof}

Now we establish the main result of this section:

\begin{proof}[Proof of Theorem \ref{thm:moment-id}]
We may assume that $\alpha_r$ have the same initial conditions as $\alpha$. Set
$L(r):=|\alpha_r(0)\alpha_r(\ell)|$,
$\rho:=\K_\alpha-\K_\beta$, 
and $\K_r:=\K_\alpha-r\rho$. By \eqref{eq:L1-continuity-estimate},
$$
|L(r)-L(s)|
\leq
|\alpha_r(\ell)\alpha_s(\ell)|
\leq
C\|\K_r-\K_s\|_{L^1([0,\ell])}
=
C|r-s|\|\rho\|_{L^1([0,\ell])}.
$$
Thus $L(r)$ is Lipschitz. It suffices now to show that $L'(r)=\MA(\alpha_r,\K_\alpha-\K_\beta)$ for almost every $r\in[0,1]$.
 First suppose that $\alpha_r(0)=\alpha_r(\ell)$. Since $L\geq0$ and
$L(r)=0$, the point $r$ is a minimum point of $L$. Hence $L'(r)=0$. By
definition, $\MA(\alpha_r,\rho)=0$ as well. So it remains to consider the case where $\alpha_r(0)\neq\alpha_r(\ell)$.

 First
assume that $d\rho$ is a finite sum of atoms,
$
d\rho=\sum_i c_i\delta_{t_i},
$
where $\delta_{t_i}$ is the unit point mass at $t_i$.
Changing $r$ to $r+h$ replaces $\K_r$ by $\K_r-h\rho$. Thus  the atom
$d\K_r(\{t_i\})$  changes by $-c_i h$. Keeping the other atoms fixed, this means that the tail $\alpha_r|_{[t_i,\ell]}$ rotates about
$\alpha_r(t_i)$ by the angle $-c_i h$, which is the same as the rotation in  the hinge
$\alpha_r(0)\alpha_r(t_i)\alpha_r(\ell)$. Hence, by Lemma~\ref{lem:hinge-derivative}, the contribution to
$L'(r)$ is
$
c_i\,m(\alpha_r(0)\alpha_r(t_i)\alpha_r(\ell)).
$
Summing over the atoms gives
$$
L'(r)=\sum_i c_i\,m\big(\alpha_r(0)\alpha_r(t_i)\alpha_r(\ell)\big)
=\MA(\alpha_r,\rho).
$$

For the general case, let
$
P_j=\{0=t_0^j<\cdots<t_{N_j}^j=\ell\}
$
be a sequence of partitions with mesh
$
|P_j|:=\max_i(t_i^j-t_{i-1}^j)\to0.
$
Let $d\rho_j$ be the signed
atomic measure which assigns the mass
$
d\rho\big((t_{i-1}^j,t_i^j]\big)
$
to a point in $(t_{i-1}^j,t_i^j)$ which is not an atom of $d\K_r$.
Let $\rho_j$ be the corresponding cumulative function. Then
$
|\rho_j(t)-\rho(t)|
\leq
|d\rho|((t_{i-1}^j,t_i^j]).
$
Hence
$$
\|\rho_j-\rho\|_{L^1([0,\ell])}
\leq
\sum_i (t_i^j-t_{i-1}^j)\,|d\rho|((t_{i-1}^j,t_i^j])
\leq
|P_j|\,|d\rho|([0,\ell])\to 0.
$$
Furthermore, $d\rho_j\to d\rho$ weakly. Indeed, if $f$ is continuous, then
$
\left|
\int_0^\ell f\,d\rho_j-\int_0^\ell f\,d\rho
\right|
$
is bounded above by
$
\epsilon_j|d\rho|([0,\ell]),
$
where
$
\epsilon_j:=\sup\{|f(t)-f(s)|\mid |t-s|\leq|P_j|\}\to0.
$

For fixed $j$ and $|h|$ sufficiently small, let $\alpha_{r,h}^j$ be the
unit speed curve with the same initial conditions as $\alpha_r$, whose curvature function is
$
\K_r-h\rho_j.
$
Since the atoms of $d\rho_j$ were chosen away from the atoms of $d\K_r$, and
since $|h|$ is small, the jumps of $\K_r-h\rho_j$ lie in $(-\pi,\pi]$.
Thus $\alpha_{r,h}^j$ exists by Proposition~\ref{prop:prescribed-curvature}.
By \eqref{eq:L1-continuity-estimate}, 
$$
|\alpha_{r,h}^j(\ell)\alpha_{r+h}(\ell)|
\leq
C\|\K_{\alpha_{r,h}^j}-\K_{\alpha_{r+h}}\|_{L^1([0,\ell])}
=
C|h|\|\rho_j-\rho\|_{L^1([0,\ell])}.
$$
Set
$
p:=\alpha_r(0)=\alpha_{r,h}^j(0)=\alpha_{r+h}(0),
$
and let
$
L_j(h)=|p\alpha_{r,h}^j(\ell)|.
$
Then, by the reverse triangle inequality,
$$
|L_j(h)-L(r+h)|
=
\big||p\alpha_{r,h}^j(\ell)|-|p\alpha_{r+h}(\ell)|\big|
\leq
|\alpha_{r,h}^j(\ell)\alpha_{r+h}(\ell)|.
$$
Also $L_j(0)=L(r)$. Hence, for $h\neq0$,
$$
\left|
\frac{L_j(h)-L_j(0)}{h}
-
\frac{L(r+h)-L(r)}{h}
\right|
\leq
C\|\rho_j-\rho\|_{L^1([0,\ell])}.
$$
At every differentiable point $r$, we obtain
$
\left|L_j'(0)-L'(r)
\right|\leq C\|\rho_j-\rho\|_{L^1([0,\ell])}.
$
Since $\rho_j\to\rho$ in $L^1([0,\ell])$, it follows that $L_j'(0)\to L'(r)$. For each $j$, the atomic case gives
$$
L_j'(0)
=
\int_0^\ell
m\big(\alpha_r(0)\alpha_r(t)\alpha_r(\ell)\big)\,d\rho_j(t).
$$
Since $\alpha_r(0)\neq\alpha_r(\ell)$, the function
$
t\mapsto m\big(\alpha_r(0)\alpha_r(t)\alpha_r(\ell)\big)
$
is continuous. Since $d\rho_j\to d\rho$ weakly, it follows that
$$
\int_0^\ell
m\big(\alpha_r(0)\alpha_r(t)\alpha_r(\ell)\big)\,d\rho_j(t)
\to
\int_0^\ell
m\big(\alpha_r(0)\alpha_r(t)\alpha_r(\ell)\big)\,d\rho(t).
$$
Combining this with $L_j'(0)\to L'(r)$ gives
$$
L'(r)=
\int_0^\ell
m\big(\alpha_r(0)\alpha_r(t)\alpha_r(\ell)\big)\,d\rho(t)
=
\MA(\alpha_r,\rho).
$$
By Proposition~\ref{prop:prescribed-curvature}, $\alpha_0=\alpha$, while
$\alpha_1$ is congruent to $\beta$. Since $L$ is Lipschitz, it is absolutely continuous. Hence
$$
|\beta(0)\beta(\ell)|-|\alpha(0)\alpha(\ell)|
=
L(1)-L(0)
=
\int_0^1 L'(r)\,dr
=
\int_0^1\MA(\alpha_r,\K_\alpha-\K_\beta)\,dr.
$$
\end{proof}

\section{Comparison in Model Planes}\label{sec:model-plane}
Here we apply the moment arm identity \eqref{eq:moment-id} and the other observations thus far to establish  Theorem \ref{thm:main} for model planes:

\begin{proposition}\label{prop:main-model}
Let $\alpha$, $\beta \colon [0,\ell] \to \M^2_k$  be unit speed convex curves. Suppose that $\kappa_\beta \leq \kappa_\alpha$.
Then the endpoint inequality \eqref{eq:schur} holds. Furthermore, 
equality holds  in \eqref{eq:schur} only if the map $\alpha(t)\mapsto\beta(t)$ extends to an isometry 
between $\conv(\alpha)$ and $\conv(\beta)$.
\end{proposition}

One might expect that when $\alpha$ and $\beta$ in Theorem \ref{thm:moment-id} are convex, then the interpolating curves $\alpha_r$ are convex as well. This would quickly yield the endpoint inequality, by ensuring that the moment arm integral is nonnegative. However, assuming that  $\alpha_r$ are convex would be an error analogous to the one in Cauchy's incorrect proof of the arm lemma \cite{sabitov2004,schoenberg-zaremba1967}. Figure \ref{fig:arms} shows a counterexample. 
\begin{figure}[h]
\begin{overpic}[height=0.75in]{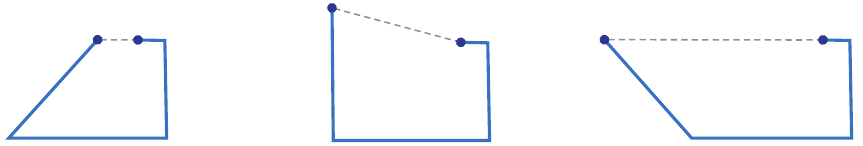}
\put(6,-2){\small $\alpha_0=\alpha$}
\put(45,-2){\small $\alpha_{\frac{1}{2}}$}
\put(84,-2){\small $\alpha_{1}=\beta$}
\end{overpic}
\caption{}\label{fig:arms}
\end{figure}
In general $\alpha_r$ are not convex when the angles which $\alpha$ makes with its endpoint geodesic segment are obtuse.
Consequently, the 
proof of Proposition \ref{prop:main-model} becomes considerably more involved and subtle in the general case. To establish this result, first we gather some basic facts about locally convex curves and their curvature measure (Section \ref{subsec:local-convex}). Using these observations, we then establish the endpoint inequality in the case where $\alpha(0)\neq\alpha(\ell)$ via a variational argument (Section \ref{subsec:nonclosed}), and finally consider the general case (Section \ref{subsec:general}).

\subsection{Locally convex curves}\label{subsec:local-convex}
A curve $\alpha\colon[0,\ell]\to\M^2_k$ is \emph{locally convex} if every point $t\in [0,\ell]$ has a neighborhood $U_t\subset[0,\ell]$ such that $\alpha|_{U_t}$ is injective, and $\alpha(U_t)$ lies on the boundary of a convex body which is always on the same side of $\alpha(U_t)$, i.e.,  to the right or to the left. Here the \emph{left side} is where $J(T)$ points  at every smooth point of $\alpha(U_t)$, and the other side is the \emph{right side}. If the convex bodies lie to the left of $\alpha(U_t)$ we say that $\alpha$ \emph{turns left}.

\begin{lemma}\label{lem:local-convexity}
Let $\alpha\colon[0,\ell]\to\M^2_k$ be a unit speed curve. Then $\alpha$ is locally convex and turns left if and only if 
$d\K_\alpha\geq0$ and $d\K_\alpha(\{t\})<\pi$ for all $t\in[0,\ell]$. 
\end{lemma}
\begin{proof}
As we remarked in Section \ref{sec:moment-arm}, if $\alpha$ is locally convex and  turns left, then $d\K_\alpha\geq 0$. Furthermore, 
$d\K_\alpha(\{t\})<\pi$ since a convex curve cannot have cusps. Conversely, suppose that $d\K_\alpha\geq0$ and $d\K_\alpha(\{t\})<\pi$ for all $t\in[0,\ell]$. Then we may choose a
closed interval $I\ni t_0$ so small that
$
d\K_\alpha(I)<\pi.
$
If $k>0$, we choose $I$ still smaller so that $\alpha(I)$ lies in a convex
geodesic ball $B$ with
$
k|B|<\pi-d\K_\alpha(I).
$
This is possible because $d\K_\alpha(I)\to d\K_\alpha(\{t_0\})$ as $I$
shrinks to $t_0$, while $|B|\to 0$. We claim that $\alpha|_I$ is injective. Suppose not. Then there exist $[s,t]\subset I$ such that $\alpha|_{[s,t)}$ is injective and $\alpha(t)=\alpha(s)$.  Let $\Omega$ be a compact domain bounded by $\alpha([s,t])$, and let
$\phi:=\angle\big(\alpha'_+(s),\alpha'_-(t)\big)$ be the exterior angle at
$\alpha(s)=\alpha(t)$. By the Gauss-Bonnet theorem, if $k\leq 0$
$$
d\K_\alpha([s,t])\geq d\K_\alpha((s,t))\geq 2\pi-\phi-k|\Omega|\geq \pi-k|\Omega|\geq \pi,
$$
which is a contradiction.
If $k>0$, then choosing $\Omega\subset B$,
$$
d\K_\alpha([s,t])\geq \pi-k|\Omega|\geq\pi-k|B|>d\K_\alpha(I),
$$
again a contradiction. Thus $\alpha|_I$ is indeed injective.

It remains to show that  $\alpha(I)$
lies on the boundary of a convex body, which is on the left side of $\alpha(I)$. 
First assume that $k=0$.  By applying Proposition~\ref{prop:prescribed-curvature} and Lemma~\ref{lem:continuous} to atomic approximations of $d\K_\alpha|_I$, there are polygonal curves $\alpha_m$ with nonnegative atomic curvature measures $\mu_m$, such that
$\mu_m(I)\leq d\K_\alpha(I)<\pi$ and $\alpha_m\to\alpha|_I$ uniformly; see Section~\ref{subsec:construction} for details of such constructions. Since $\mu_m\geq0$, each $\alpha_m$ turns
left at every vertex. Its total curvature is less than $\pi$, so it is convex,
with $\conv(\alpha_m)$ lying on its left side. Passing to a subsequence,
$\conv(\alpha_m)$ converges in the Hausdorff sense to a convex body $C$, by
Blaschke's selection principle \cite[Thm. 7.3.8]{bbi2001}. It follows that
$\alpha(I)\subset\partial C$, and $C$ lies on the left side of $\alpha(I)$.

If $k\neq 0$, let $\varphi\colon U\to\R^2$ be a projective chart centered at
$\alpha(t_0)$, i.e., the Beltrami-Klein model if $k<0$ and the gnomonic
projection if $k>0$, normalized so that $d\varphi_{\alpha(t_0)}$ is a linear
isometry. As discussed above, there are polygonal curves
$\alpha_m$ with atomic curvature measures $\mu_m\geq0$,
$\mu_m(I)\leq d\K_\alpha(I)<\pi$, such that $\alpha_m\to\alpha|_I$ uniformly.
Since $\varphi$ preserves geodesics and orientation, the curves
$\ol\alpha_m:=\varphi\circ\alpha_m$ are polygonal and turn left at all vertices. Furthermore, the angle between a pair of directions at
$p\in U$ is distorted by $d\varphi_p$ by a factor of at most $\lambda(p)$, where
$\lambda$ is continuous and $\lambda(\alpha(t_0))=1$. So we may shrink $I$ so
that $\lambda\leq1+\epsilon$ on a neighborhood of $\alpha(I)$, where
$(1+\epsilon)\,d\K_\alpha(I)<\pi$. Then, for large $m$,
$$
d\K_{\ol\alpha_m}(I)\leq(1+\epsilon)\,\mu_m(I)
\leq(1+\epsilon)\,d\K_\alpha(I)<\pi.
$$
Hence $\ol\alpha_m$ are convex and turn left, and as in the case $k=0$ it
follows that $\ol\alpha(I):=\varphi\circ\alpha(I)$ lies on the boundary of a convex body
$\ol C\subset\R^2$ which lies on its left side. Since $\varphi$ preserves
orientation and convexity, $\alpha(I)\subset\partial\,\varphi^{-1}(\ol C)$, and
$\varphi^{-1}(\ol C)$ is a convex body in $\M^2_k$ lying on the left side of
$\alpha(I)$.
\end{proof}

A curve $\alpha\colon[0,\ell]\to\M^2_k$ is \emph{closed} if $\alpha(0)=\alpha(\ell)$, and is \emph{locally convex} if its periodic extension $\alpha\colon\R/(\ell\mathbf{Z})\to\M^2_k$ is locally convex. 
We say that $\alpha$ is \emph{simple} if it is injective. 
If $\alpha$ is closed and is injective on $[0,\ell)$ we say that $\alpha$ is a \emph{simple closed curve}.
It is easy to see that simple closed locally convex curves are convex. Indeed suppose that $\alpha$ turns left, and let $\Omega$ be the domain bounded by $\alpha$ which lies on the left side of it. Let $p_0$, $p_1$ be a pair of points in the interior $\inte(\Omega)$, and $\gamma\colon[0,1]\to \inte(\Omega)$ be a curve with $\gamma(0)=p_0$, $\gamma(1)=p_1$. Let $\ol t\in[0,1]$ be the supremum of $t$ such that the geodesic segment $\gamma(0)\gamma(t)\subset \inte(\Omega)$. If $\ol t=1$ then $\Omega$ is convex and we are done. Otherwise, for some $t$, $\gamma(0)\gamma(t)$ intersects $\partial\Omega=\alpha$ at an interior point, which forces $\gamma(0)\gamma(t)\subset \partial\Omega$, which is a contradiction. The next observation extends this characterization to nonclosed curves.

\begin{lemma}\label{lem:convex}
Let $\alpha\colon[0,\ell]\to\M^2_k$ be a locally convex simple curve. Suppose that $\alpha(0)$ and $\alpha(\ell)$ lie on $\partial\conv(\alpha)$. Then $\alpha$ is convex.
\end{lemma}
\begin{proof}
If $\alpha$ is a geodesic then there is nothing to prove, so we may assume that $\conv(\alpha)$ has interior points. Then there are two arcs in $\partial\conv(\alpha)$ determined by $\alpha(0)$ and $\alpha(\ell)$. Orient each arc from $\alpha(\ell)$ to $\alpha(0)$, see Figure \ref{fig:ellipse}. Then one of the arcs turns left and the other turns right. We may assume that $\alpha$ turns left.   Attach the arc which turns left, say $A$, to the endpoints of $\alpha$.  Then we obtain a closed locally convex curve $\ol\alpha$ by Lemma \ref{lem:local-convexity}. If $\ol\alpha$ is simple then it is convex, as discussed above, and we are done. Otherwise, there exists $t\in (0,\ell)$ such that $\alpha(t)$ lies in the interior of $A$. 
\begin{figure}[h]
\begin{overpic}[height=0.9in]{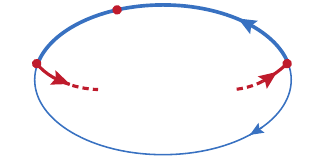}
\put(-5,28){\small $\alpha(0)$}
\put(90,28){\small $\alpha(\ell)$}
\put(33,40){\small $\alpha(t)$}
\put(68,46){\small $A$}
\end{overpic}
\caption{}\label{fig:ellipse}
\end{figure}
Since both $A$ and $\alpha$ turn left, and locally support each other at $\alpha(t)$, it follows that there is an interval $I\subset[0,\ell]$ containing $t$ in its interior such that $\alpha(I)\subset A$. Indeed, using a projective chart, we may assume that we are in $\R^2$, and 
 represent $A$ and $\alpha$ near $\alpha(t)$ as the graphs of a convex function $f$ and a concave function $g$, with $f\leq g$ which yields that $f=g$. So $\alpha(I)$ is a geodesic segment. Repeating this argument at an endpoint of $I$, if it lies in $(0,\ell)$, shows that the maximal interval $I=[0,\ell]$. So $\alpha$ is a geodesic, which is a contradiction. Hence $\ol\alpha$ is indeed simple, which completes the argument.
\end{proof}

\subsection{The nonclosed case}\label{subsec:nonclosed}
Here we use the observations of the last subsection to establish Proposition \ref{prop:main-model} in the case where $\alpha$ is not closed. First, we need the following fact:

\begin{lemma}\label{lem:corners}
Let $\alpha,\beta\colon[0,\ell]\to\M^2_k$ be unit speed convex curves with
$\kappa_\beta\leq \kappa_\alpha$. Suppose that
there exist $0\leq a<b\leq\ell$ such that $\alpha|_{[a,b]}$ and
$\beta|_{[a,b]}$ are congruent, 
$\beta([0,a])\cup\beta([b,\ell])$ lies on the geodesic through the endpoints
of $\beta$, and $\beta$ is not a geodesic.  Then $\alpha$ and $\beta$ are congruent.
\end{lemma}

\begin{proof}
After composing with isometries of $\M^2_k$, we may assume that 
$
\alpha|_{[a,b]}=\beta|_{[a,b]},
$
and $d\K_\alpha$, $d\K_\beta\geq 0$.
Set $C:=\conv(\alpha([a,b]))=\conv(\beta([a,b]))$. 
Let $\L$ be the geodesic through $\beta(0)$ and $\beta(\ell)$. Since $\beta$ is not a geodesic, $C$ has interior points and thus lies on a unique side of $\L$ which we denote by $\L^+$. It suffices to show that 
$$
\alpha([0,a])\cup\alpha([b,\ell])\subset \L.
$$
 There is nothing to prove if $a=0$ and $b=\ell$. Suppose that $a>0$. Since $\alpha$ and $\beta$ agree on $[a,b]$, $\alpha'_+(a)=\beta'_+(a)$. Thus if $\beta([a,b])$ is tangent to $\L$ at $a$, then the same holds for $\alpha([a,b])$. It follows that $\alpha([0,a])\subset\L^+$. But $\alpha([0,a])$ cannot enter the interior of $C$, since $\alpha$ is convex. Thus $\alpha([0,a])\subset\L$.  If $\beta([a,b])$ is not tangent to $\L$ at $a$, then $\beta$ has a corner at $a$, with angle $\psi_\beta$. 
Then $d\K_{\beta}(\{a\})=\pi-\psi_\beta>0$. By assumption $d\K_{\alpha}\geq d\K_{\beta}$. So $d\K_{\alpha}(\{a\})>0$ which means that $\alpha$ has a corner at $a$ as well, with angle $\psi_\alpha$. Then
$$
\pi-\psi_\alpha=d\K_{\alpha}(\{a\})\geq d\K_{\beta}(\{a\})=\pi-\psi_\beta.
$$
So $\psi_\alpha\leq\psi_\beta$. If $\psi_\alpha <\psi_\beta$, then $\alpha([0,a])$ enters the interior of $C$ which is not possible.
Hence $\psi_\alpha=\psi_\beta$. So $\alpha([0,a])$ is tangent to $\L$ at $a$, which again yields that $\alpha([0,a])\subset\L$. The case where $b<\ell$ is treated similarly.
\end{proof}

 Using the last three lemmas  we next establish the endpoint inequality in the following special case, by employing the moment arm identity and a minimization argument:

\begin{lemma}\label{lem:nonclosed-model}
Let $\alpha$, $\beta \colon [0,\ell]\to\M^2_k$ be unit speed convex curves with
$\alpha(0)\ne\alpha(\ell)$. Suppose that
$
\kappa_\beta\leq\kappa_\alpha.
$
Then the endpoint inequality \eqref{eq:schur} holds.
Furthermore, if equality holds in \eqref{eq:schur}, then $\alpha$ and $\beta$ are congruent.
\end{lemma}

\begin{proof}
After composition with  isometries of $\M^2_k$, we may assume that $\beta$ has the same initial
conditions as $\alpha$, and $d\K_\alpha$, $d\K_\beta\geq 0$. Let $\mathcal C$ be
the family of unit speed convex curves $\gamma\colon[0,\ell]\to\M^2_k$ with the same initial conditions as 
$\alpha$ and satisfying
$
0\leq d\K_\gamma\leq d\K_\alpha.
$
We claim that there is a curve $\gamma_0\in\mathcal C$ with minimal endpoint distance. Then we
suppose, towards a contradiction, that
\be\label{eq:gamma0-alpha}
|\gamma_0(0)\gamma_0(\ell)|<|\alpha(0)\alpha(\ell)|.
\ee
The argument will be presented in four parts. First we establish the existence of $\gamma_0$, then we consider \eqref{eq:gamma0-alpha} in the case where $\gamma_0$ is not closed, followed by the general case. Finally we characterize the equality case in the endpoint inequality.

\smallskip
(\emph{Part I}) To see that $\gamma_0$ exists, let
$\gamma_m\in\mathcal C$ be a minimizing sequence for the endpoint distance.
The functions $\K_{\gamma_m}$ and $\K_\alpha-\K_{\gamma_m}$ are nondecreasing
and uniformly bounded. So, by Helly's selection theorem \cite{rudin1976}, after
passing to a subsequence, $\K_{\gamma_m}$ converges pointwise to a function
$\K$ such that $\K$ and $\K_\alpha-\K$ are again nondecreasing. Hence the
increments of $\K$ are nonnegative and bounded by those of $\K_\alpha$; since
$\K_\alpha$ is left-continuous, so is $\K$, and $0\leq d\K\leq d\K_\alpha$. In particular $d\K(\{0\})=0=d\K(\{\ell\})$, and $d\K(\{t\})<\pi$ on 
$(0,\ell)$, since $\alpha$ is convex. By
Proposition~\ref{prop:prescribed-curvature}, there is a unit speed curve
$\gamma_0$ with the same initial conditions as $\alpha$ and
$d\K_{\gamma_0}=d\K$. By the dominated convergence theorem and
Lemma~\ref{lem:continuous}, $\gamma_m\to\gamma_0$ uniformly. So $\gamma_0$
minimizes the endpoint distance. It remains to show that $\gamma_0$ is convex.

Since $\gamma_0$ is unit speed, $\conv(\gamma_0)$ cannot be a single point. Thus, if $\conv(\gamma_0)$ has empty interior, then it has to be a geodesic segment covered by $\gamma_0$. By Lemma \ref{lem:local-convexity}, $\gamma_0$ is locally one-to-one. So $\gamma_0$ is a geodesic segment. In particular it is convex and we are done. Assume then that $\conv(\gamma_0)$ has interior points. Since $\conv(\gamma_m)\to \conv(\gamma_0)$
in the Hausdorff distance, and $\gamma_m\subset \partial\conv(\gamma_m)$, it follows that $\gamma_0\subset \partial\conv(\gamma_0)$. Since $|\partial\conv(\gamma_m)|\geq\length(\gamma_m)=\ell$, we have $|\partial\conv(\gamma_0)|\geq\ell$. Thus $\gamma_0$ is one-to-one on $(0,\ell)$, since it is locally one-to-one, and we conclude again that  $\gamma_0$ is convex.

\smallskip
(\emph{Part II}) Assume that $\gamma_0$ is not closed. Then the endpoint geodesic 
$\L=\L_{\gamma_0(0)\gamma_0(\ell)}$ is well-defined.  Put
$
d\lambda:=d\K_\alpha-d\K_{\gamma_0}.
$
If $d\lambda\equiv 0$, then $\gamma_0$ and $\alpha$ are congruent by
Proposition~\ref{prop:prescribed-curvature}, contrary to \eqref{eq:gamma0-alpha}. Thus $d\lambda\not\equiv0$. Also note that $\gamma_0$
is not a geodesic, since otherwise
$
|\gamma_0(0)\gamma_0(\ell)|=\ell\geq|\alpha(0)\alpha(\ell)|.
$
  We claim that 
$$
\MA(\gamma_0,\lambda)>0.
$$
 Suppose to the contrary that $\MA(\gamma_0,\lambda)=0$. 
 Since $d\K_{\gamma_0}\geq0$, the curve $\gamma_0$ turns left, as noted in
Section~\ref{sec:moment-arm}. So the moment arm at every point of
$\gamma_0$ not lying on $\L$ is positive. Thus
$d\lambda$ is supported on $\gamma_0^{-1}(\L)$. Since $\gamma_0$ is
convex and is not a geodesic, there exist $0\leq a_0<b_0\leq\ell$ such that
$\gamma_0([0,a_0])\cup\gamma_0([b_0,\ell])$ lies on $\L$, while
$\gamma_0((a_0,b_0))$ lies off $\L$, see Figure \ref{fig:lasso}. Hence $d\lambda=0$ on $(a_0,b_0)$, and
so $\gamma_0$ and $\alpha$ have the same curvature measure on $(a_0,b_0)$. By
Proposition~\ref{prop:prescribed-curvature}, $\gamma_0|_{[a_0,b_0]}$ and
$\alpha|_{[a_0,b_0]}$ are congruent. Lemma~\ref{lem:corners}, applied with
$\beta=\gamma_0$, shows that $\alpha$ and $\gamma_0$ are congruent. This
contradicts \eqref{eq:gamma0-alpha}. So
$\MA(\gamma_0,\lambda)>0$ as claimed.

\begin{figure}[h]
\begin{overpic}[height=0.85in]{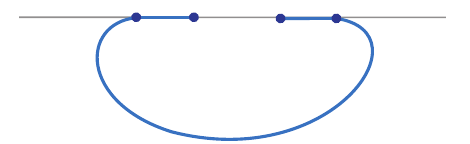}
\put(20,6){\small $\gamma_0$}
\put(0,26){\small $\L$}
\put(19,30){\small $\gamma_0(a_0)$}
\put(36,30){\small $\gamma_0(0)$}
\put(55,30){\small $\gamma_0(\ell)$}
\put(70,30){\small $\gamma_0(b_0)$}
\end{overpic}
\caption{}\label{fig:lasso}
\end{figure}

For $s\geq0$, let $\gamma_s$ be the unit speed curve with the same initial
conditions as $\gamma_0$ and curvature measure
$
d\K_{\gamma_s}:=d\K_{\gamma_0}+s\,d\lambda.
$
Since $\MA(\gamma_0,\lambda)>0$, we have 
$
\MA(\gamma_s,\lambda)>0
$
for  small $s>0$ by continuity. Therefore, by Theorem~\ref{thm:moment-id},
with the sign reversed since curvature is being added,
$$
|\gamma_s(0)\gamma_s(\ell)|-|\gamma_0(0)\gamma_0(\ell)|
=
-\int_0^s\MA(\gamma_u,\lambda)\,du<0
$$
for small $s>0$. If $\gamma_s$ is convex, then $\gamma_s\in\mathcal C$ and
$|\gamma_s(0)\gamma_s(\ell)|<|\gamma_0(0)\gamma_0(\ell)|$, contradicting the choice of $\gamma_0$. 
So suppose that $\gamma_s$ is not convex for any small
$s>0$.

Since $\gamma_0$ is simple, we may choose $s$ so small that $\gamma_s$ is simple. Indeed, since
$\gamma_s\to\gamma_0$ uniformly by Lemma~\ref{lem:continuous},  any self-intersection loop of $\gamma_s$ must
be small, but since $d\K_{\gamma_s}\leq d\K_\alpha$, small loops are
ruled out by the Gauss-Bonnet argument in the proof of
Lemma~\ref{lem:local-convexity}.
Moreover, since 
$$
d\K_{\gamma_s}=(1-s)\,d\K_{\gamma_0}+s\,d\K_\alpha,
$$
$d\K_{\gamma_s}$ is nonnegative
with atoms of mass less than $\pi$, since $\alpha$ and $\gamma_0$ are convex.
So $\gamma_s$ is locally convex by Lemma~\ref{lem:local-convexity}.
Since $\gamma_s$ is not convex, at least one of the endpoints of $\gamma_s$ lies in the interior of $\conv(\gamma_s)$ by Lemma \ref{lem:convex}; see Figure \ref{fig:arcs}.  
\begin{figure}[h]
\begin{overpic}[height=0.9in]{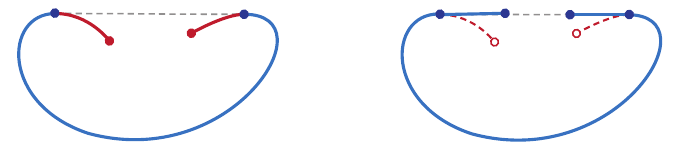}
\put(5,2){\small $\gamma_s$}
\put(13,12){\small $\gamma_s(0)$}
\put(26,13){\small $\gamma_s(\ell)$}
\put(4.5,21){\small $\gamma_s(a_s)$}
\put(31,21){\small $\gamma_s(b_s)$}
\put(61,2){\small $\ol\gamma_s$}
\put(70,21){\small $\ol\gamma_s(0)$}
\put(80,21){\small $\ol\gamma_s(\ell)$}
\end{overpic}
\caption{}\label{fig:arcs}
\end{figure}
Let $\gamma_s([0,a_s))$, $\gamma_s((b_s,\ell])$ be the maximal  subarcs  in the interior of $\conv(\gamma_s)$ for $0\leq a_s<b_s\leq\ell$. Note that $\gamma_s|_{[a_s,b_s]}$ is convex by Lemma~\ref{lem:convex}, since its
endpoints lie on $\partial\conv(\gamma_s)=\partial\conv(\gamma_s([a_s,b_s]))$.
Hence $\partial\conv(\gamma_s)$ consists of $\gamma_s([a_s,b_s])$ together with
the segment $\gamma_s(a_s)\gamma_s(b_s)$. 

Since $\gamma_s\to\gamma_0$ uniformly, by Lemma \ref{lem:continuous},
the convex hulls converge, and therefore the perimeters
$
|\partial\conv(\gamma_s)|\to|\partial\conv(\gamma_0)|.
$
But $\gamma_0$ is convex and nonclosed, so
$
|\partial\conv(\gamma_0)|>\ell.
$
Thus, for $s$ sufficiently small,
$$
|\gamma_s(a_s)\gamma_s(b_s)|
=
|\partial\conv(\gamma_s)|-(b_s-a_s)
>
\ell-(b_s-a_s)
=
a_s+(\ell-b_s).
$$
So we may replace $\gamma_s([0,a_s))$ and $\gamma_s((b_s,\ell])$ by subarcs of
the same length on $\gamma_s(a_s)\gamma_s(b_s)$, adjacent to $\gamma_s(a_s)$ and
$\gamma_s(b_s)$ respectively (as shown in the right diagram in Figure \ref{fig:arcs}).  The resulting curve $\ol\gamma_s$ is simple and locally convex. Hence it is convex by Lemma \ref{lem:convex}. It also has length $\ell$, and satisfies
$$
d\K_{\ol\gamma_s}\leq d\K_{\gamma_s}\leq d\K_\alpha.
$$
Indeed the measures
agree on $(a_s,b_s)$, vanish or dominate on the flat parts, and at $a_s$, $b_s$
any possible corner angle of $\ol\gamma_s$ is at least that of $\gamma_s$, which yields that the exterior angle of $\ol\gamma_s$ is not greater than that of $\gamma_s$.
By construction, $|\gamma_s(a_s)\ol\gamma_s(0)|=a_s$, and $|\ol\gamma_s(\ell)\gamma_s(b_s)|=\ell-b_s$. Thus, since $\ol\gamma_s(a_s)=\gamma_s(a_s)$, $\ol\gamma_s(0)$, $\ol\gamma_s(\ell)$, and $\ol\gamma_s(b_s)=\gamma_s(b_s)$ all lie on the same line in that order, we have
\begin{multline*}
a_s+|\ol\gamma_s(0)\ol\gamma_s(\ell)|+(\ell-b_s)
=
|\ol\gamma_s(a_s)\ol\gamma_s(b_s)|
=
|\gamma_s(a_s)\gamma_s(b_s)|
\leq\\
|\gamma_s(a_s)\gamma_s(0)|+|\gamma_s(0)\gamma_s(\ell)|+|\gamma_s(\ell)\gamma_s(b_s)|
\leq
a_s+|\gamma_s(0)\gamma_s(\ell)|+(\ell-b_s).
\end{multline*}
So we obtain
$
|\ol\gamma_s(0)\ol\gamma_s(\ell)|
\leq
|\gamma_s(0)\gamma_s(\ell)|
<
|\gamma_0(0)\gamma_0(\ell)|.
$
After composing with an isometry, we may assume that $\ol\gamma_s$ has the same initial conditions as $\alpha$.
Thus $\ol\gamma_s\in\mathcal C$, again contradicting the choice of
$\gamma_0$. So the endpoint inequality holds when
$\gamma_0(0)\ne\gamma_0(\ell)$.

\smallskip
(\emph{Part III})
Now we remove the assumption that $\gamma_0(0)\ne\gamma_0(\ell)$, which we achieve by applying the above discussion to subarcs $\alpha|_{[0,t]}$. Let $\mathcal C_t$ be the corresponding class of curves, and  $\gamma_0^t\in\mathcal C_t$ be a minimizer for endpoint
distance. Let 
$$
A:=\Big\{t\in(0,\ell]\colon \left|\gamma_0^t(0)\gamma_0^t(t)\right|=|\alpha(0)\alpha(t)|\Big\}.
$$
Then $t\in A$ if and only if $\alpha|_{[0,t]}$ satisfies the endpoint inequality, which is the case as soon as $\gamma_0^t(0)\neq\gamma_0^t(t)$. The set $A$ contains a small interval $(0,\delta)$. Indeed, if
$\gamma_0^t$ were closed for a sequence $t\to0$, then its total curvature $\K_{\gamma_0^t}(t)$ would tend to at least $\pi$, as discussed in the proof of Lemma \ref{lem:local-convexity}, while
$
\K_{\gamma_0^t}(t)\leq\K_\alpha(t)\to0.
$

Put $\ol t:=\sup A$.  Then $\ol t\in A$. Indeed
if $t_i\in A$ and $t_i\to\ol t^{\,-}$, then
$\gamma_0^{\ol t}|_{[0,t_i]}\in\mathcal C_{t_i}$, and so
$
|\gamma_0^{\ol t}(0)\gamma_0^{\ol t}(t_i)|
\geq
|\alpha(0)\alpha(t_i)|.
$
Letting $i\to\infty$ gives one inequality, and the reverse one follows from
$\alpha|_{[0,\ol t]}\in\mathcal C_{\ol t}$. Suppose  that $\ol t<\ell$. Choose $\epsilon>0$ so small that
$\ol t+\epsilon\leq\ell$ and
$
\epsilon<|\alpha(0)\alpha(\ol t)|.
$
If
$
\gamma_0^{\ol t+\epsilon}(0)=\gamma_0^{\ol t+\epsilon}(\ol t+\epsilon),
$
then
$
|\gamma_0^{\ol t+\epsilon}(0)\gamma_0^{\ol t+\epsilon}(\ol t)|
\leq
\epsilon
<
|\alpha(0)\alpha(\ol t)|,
$
contradicting $\ol t\in A$. Hence $\gamma_0^{\ol t+\epsilon}$ is not closed,
and therefore $\ol t+\epsilon\in A$, a contradiction.
Thus $\ol t=\ell$, which completes the proof of the desired inequality.

\smallskip
(\emph{Part IV})
Finally suppose that $|\beta(0)\beta(\ell)|=|\alpha(0)\alpha(\ell)|$.
Then $\beta$ is a minimizer in $\mathcal C$, and $\beta$ is not closed,
since $\alpha$ is not. Suppose, towards a contradiction, that
$d\lambda:=d\K_\alpha-d\K_\beta\neq0$. Then $\beta$ is not a geodesic;
otherwise $\ell=|\beta(0)\beta(\ell)|=|\alpha(0)\alpha(\ell)|$, so
$\alpha$ would be a geodesic as well, and $d\lambda=0$. So we may run
the argument above with $\gamma_0=\beta$. If $\MA(\beta,\lambda)>0$,
we obtain a curve in $\mathcal C$ with endpoint distance strictly less
than the minimum $|\alpha(0)\alpha(\ell)|$, which is impossible. If
$\MA(\beta,\lambda)=0$, then Lemma~\ref{lem:corners} shows that
$\alpha$ and $\beta$ are congruent, whence $d\K_\beta=d\K_\alpha$,
contradicting $d\lambda\neq0$. Hence $d\lambda=0$, and $\alpha$,
$\beta$ are congruent by Proposition~\ref{prop:prescribed-curvature}.
\end{proof}

\subsection{The general case}\label{subsec:general}
Finally we establish the general case of Proposition \ref{prop:main-model} after two more observations.
For a closed unit speed curve $\alpha\colon[0,\ell]\to M^n$ in a Riemannian manifold, let
$$
\phi_\alpha:=\angle\big(\alpha'_-(\ell),\alpha'_+(0)\big)
$$
be the \emph{closing exterior angle}. The following result was established by Ni \cite{ni2023} in the piecewise smooth case. We verify that the same argument extends to the general case by approximation. 

\begin{lemma}\label{lem:ni}
Let $\alpha,\beta\colon[0,\ell]\to\S^2$ be closed convex unit speed curves. Suppose that
$
\kappa_\beta\leq \kappa_\alpha,
$
and $\phi_\beta\leq \phi_\alpha$.
Then $\alpha$ and $\beta$ are congruent.
\end{lemma}

\begin{proof}
Choose an affine plane $\Pi\subset\R^3$ which
does not pass through the origin and intersects the cone over $\alpha$ in a
convex plane curve. Write
$
\Pi=\{x\in\R^3:\langle x,u\rangle=1\},
$
where $\langle\alpha,u\rangle>0$, and set
$
R(t):=\langle\alpha(t),u\rangle^{-1}.
$
Define
$
\bar\alpha(t):=R(t)\alpha(t),
$
$        
\bar\beta(t):=R(t)\beta(t).
$
Then $\bar\alpha$ is a closed convex plane curve in $\Pi$, and $\bar\beta$ is a
closed curve in $\R^3$. Moreover, $\bar\alpha$ and $\bar\beta$ have the same
arclength function, since, for almost every $t$,
$$
 \left|\bar\alpha'(t)\right|^2=\big(R'(t)\big)^2+R(t)^2=\left|\bar\beta'(t)\right|^2.
$$

We claim that $\kappa_{\bar\beta}\leq \kappa_{\bar\alpha}$, which  holds in the piecewise smooth case \cite[pp. 111--112]{ni2023}. 
For the general case, note that since $\alpha$ and $\beta$ are convex, after composing them with reflections of $\S^2$ if necessary, we may assume that $d\K_\alpha$ and $d\K_\beta$ are nonnegative. Then $\kappa_\beta\leq\kappa_\alpha$ yields $d\K_\beta\leq d\K_\alpha$. Now let $I=[a,b]\subset(0,\ell)$ be an interval whose endpoints are not atoms of $d\K_\alpha$ or $d\K_\beta$. Put
$$
\mu:=d\K_\alpha|_I, \qquad \nu:=d\K_\beta|_I,\qquad \lambda:=\mu-\nu .
$$
Choose partitions
$
a=t_0^m<\cdots<t_{N_m}^m=b
$
with mesh tending to zero, and no partition point an atom of $\mu$ or $\nu$.
On each interval $(t_{i-1}^m,t_i^m)$ choose finite nonnegative atomic measures
$\nu_i^m$ and $\lambda_i^m$, with disjoint supports, such that every atom has
mass less than $\pi$ and
$
\nu_i^m((t_{i-1}^m,t_i^m))=\nu([t_{i-1}^m,t_i^m)),
$
$
\lambda_i^m((t_{i-1}^m,t_i^m))=\lambda([t_{i-1}^m,t_i^m)).
$
Set
$$
\nu_m:=\sum_i\nu_i^m,\qquad
\lambda_m:=\sum_i\lambda_i^m,\qquad
\mu_m:=\nu_m+\lambda_m .
$$
Then $\nu_m\leq\mu_m$, all atoms of $\mu_m$ and $\nu_m$ have mass less than
$\pi$, and the cumulative functions of $\mu_m$ and $\nu_m$ converge in
$L^1(I)$ to those of $\mu$ and $\nu$. By Proposition~\ref{prop:prescribed-curvature}, applied on $I$ with the initial data of $\alpha|_I$ and $\beta|_I$, there are piecewise smooth  curves $\alpha_m,\beta_m\colon I\to\S^2$ with curvature measures $\mu_m$, $\nu_m$. By Lemma~\ref{lem:continuous}, $\alpha_m\to\alpha|_I$ and $\beta_m\to\beta|_I$, with convergence of tangent fields in $L^1$. 
Set
$
R_m(t):=\langle\alpha_m(t),u\rangle^{-1},
$
$
\bar\alpha_m:=R_m\alpha_m,
$
$ 
\bar\beta_m:=R_m\beta_m.
$
The piecewise smooth case gives
$$
 \kappa_{\bar\beta_m}(I)\leq \kappa_{\bar\alpha_m}(I).
$$
Since $\alpha_m\to\alpha|_I$ and $\beta_m\to\beta|_I$ uniformly, and $R_m\to R$ uniformly, we have $\bar\beta_m\to\bar\beta$ uniformly on $I$. Thus, by lower semicontinuity of total curvature under uniform convergence \cite[Thm.~5.1.1]{alexandrov-reshetnyak1989,karuwannapatana-maneesawarng2007},
$$
\kappa_{\bar\beta}(I)
\leq
\liminf_{m\to\infty}\kappa_{\bar\beta_m}(I).
$$
On the other hand, $\mu_m(I)=\mu(I)$. Hence Lemma~\ref{lem:continuous} gives convergence of the terminal tangent directions of $\alpha_m$ to that of $\alpha$. Thus the endpoint tangent directions of $\bar\alpha_m$ converge to those of $\bar\alpha$. Since $d\K_{\bar\alpha_m}\geq0$, as central projection preserves the sign of $d\K$, and $a$, $b$ are not atoms of the curvature measure of $\bar\alpha$,
$
\kappa_{\bar\alpha_m}(I)\to \kappa_{\bar\alpha}(I).
$
Consequently
$
\kappa_{\bar\beta}(I)
\leq
\kappa_{\bar\alpha}(I)
$
as claimed. Furthermore, Ni's jump-angle formula \cite[(3.6)]{ni2023}, applied at the closing point, also gives
$
\phi_{\bar\beta}\leq\phi_{\bar\alpha}.
$
The rest of the argument is just as in \cite{ni2023}. Namely we obtain
$$
2\pi
\leq
\kappa_{\bar\beta}([0,\ell])+\phi_{\bar\beta}
\leq
\kappa_{\bar\alpha}([0,\ell])+\phi_{\bar\alpha}
=
2\pi,
$$
where the first inequality is by Fenchel's theorem, since $\bar\beta$ is a closed space curve, and the last equality holds since $\bar\alpha$ is a closed convex planar curve. So $\kappa_{\bar\beta}([0,\ell])+\phi_{\bar\beta}=2\pi$, which means that $\bar\beta$ is a convex planar curve as well. Furthermore, $\kappa_{\bar\beta}= \kappa_{\bar\alpha}$. Thus, by Proposition \ref{prop:prescribed-curvature}, $\bar\alpha$ and $\bar\beta$ are congruent. Hence $\alpha$ and $\beta$ are congruent.
\end{proof}

Let $\alpha$, $\beta\colon[0,\ell]\to\M^2_k$ be closed unit speed curves. We say that $\beta$ \emph{majorizes} $\alpha$ if $|\beta(s)\beta(t)|\geq |\alpha(s)\alpha(t)|$
for all $s$, $t\in [0,\ell]$. Let $|\conv(\alpha)|$, $|\conv(\beta)|$ denote the areas of the convex hulls.

\begin{lemma}\label{lem:majorize}
Let $\alpha$, $\beta\colon[0,\ell]\to\M^2_k$ be unit speed closed convex curves. Suppose that $\beta$ majorizes $\alpha$. Then $|\conv(\beta)|\geq |\conv(\alpha)|$ and $\phi_\beta\leq\phi_\alpha$.
\end{lemma}
\begin{proof}
By the generalized Kirszbraun extension theorem, due to Lang-Schroeder
\cite{akp2024,lang-schroeder1997}, the correspondence
$\beta(t)\mapsto\alpha(t)$ extends to a distance nonincreasing map
$F\colon\conv(\beta)\to\M^2_k$. Since $F$ maps $\partial\conv(\beta)$ onto
$\partial\conv(\alpha)$ with degree one, $F(\conv(\beta))\supset\conv(\alpha)$.
Thus $|\conv(\beta)|\geq |\conv(\alpha)|$. Next we have
$$
\cos\left(\frac{\phi_\beta}{2}\right)
=
\lim_{\epsilon\to0}
\frac{|\beta(\epsilon)\beta(\ell-\epsilon)|}{2\epsilon}
\geq
\lim_{\epsilon\to0}
\frac{|\alpha(\epsilon)\alpha(\ell-\epsilon)|}{2\epsilon}
=
\cos\left(\frac{\phi_\alpha}{2}\right),
$$
where the equalities follow by taking normal coordinates at $\beta(0)$ and $\alpha(0)$, and the middle inequality holds since $\beta$ majorizes $\alpha$.
Hence
$\phi_\beta\leq\phi_\alpha$.
\end{proof}

Finally we are ready to establish the main result of this section:

\begin{proof}[Proof of Proposition \ref{prop:main-model}]
Set $\rho:=\K_\alpha-\K_\beta$. After reflecting the curves if necessary, we may assume that $d\K_\alpha$ and
$d\K_\beta$ are nonnegative. Since $\kappa_\beta\leq\kappa_\alpha$, we have
$d\K_\beta\leq d\K_\alpha$, and therefore $d\rho\geq0$.
If $d\rho=0$, then $\alpha$ and $\beta$ are congruent by Proposition \ref{prop:prescribed-curvature}.
So suppose that $d\rho\ne0$. First suppose that $\alpha(0)\ne\alpha(\ell)$. Then
Lemma~\ref{lem:nonclosed-model} gives
$$
|\beta(0)\beta(\ell)|\geq|\alpha(0)\alpha(\ell)|.
$$
If equality holds, then $\alpha$ and $\beta$ are congruent by the same lemma.
Thus the desired inequality holds, and equality is impossible in this case.
It remains only to consider the case where $\alpha$ is closed. Then the
inequality is automatic, and equality can occur only when $\beta$ is closed.
Assume therefore that $\beta$ is closed. Applying the nonclosed case to the
subarcs $\alpha|_{[s,t]}$ and $\beta|_{[s,t]}$, we obtain
$
|\beta(s)\beta(t)|\geq |\alpha(s)\alpha(t)|
$
for all $0\leq s<t<\ell$. 
By Lemma \ref{lem:majorize}, the areas $|\conv(\beta)|\geq|\conv(\alpha)|$, and the closing exterior angles $\phi_\beta\leq \phi_\alpha$.
Thus, by the Gauss-Bonnet theorem, for $k\leq 0$ we have
\begin{multline*}
2\pi-k|\conv(\beta)|=\int_0^\ell d\K_\beta+\phi_\beta
\leq \int_0^\ell d\K_\alpha+\phi_\alpha\\=2\pi-k|\conv(\alpha)|\leq 2\pi-k|\conv(\beta)|.
\end{multline*}
Hence $d\K_\beta=d\K_\alpha$. So $\alpha$ and $\beta$ are congruent. 
If $k>0$, after rescaling $\M^2_k$ to $\S^2$, we obtain the same conclusion
by Lemma \ref{lem:ni}.
\end{proof}

\section{Curvature of Majorizations}\label{sec:majorization}
In order to extend the comparison theorem in the last section, from model planes to $\CAT(k)$ spaces, we need the following result which constitutes a refinement of Reshetnyak's majorization theorem.  
By a \emph{degenerate convex curve} in $\M^2_k$ we mean a 
curve which double covers a nontrivial geodesic segment. A \emph{possibly degenerate} convex
curve is one which is either convex or degenerate.
Let $\alpha\colon [0,\ell]\to X_k $  be a rectifiable curve, and $\tilde\alpha \colon [0,\ell]\to \M^2_k$ be a possibly degenerate convex curve with  
$\length(\tilde\alpha|_{[0,t]})=\length(\alpha|_{[0,t]})$ for all $t\in[0,\ell]$.
Then $\tilde\alpha$ \emph{majorizes} $\alpha$ if  for all $s$, $t\in [0,\ell]$, 
$$
|\tilde\alpha(t)\tilde\alpha(s)|\geq|\alpha(t)\alpha(s)|  \qquad\text{and}\qquad |\tilde\alpha(0)\tilde\alpha(\ell)|=|\alpha(0)\alpha(\ell)|.
$$
Reshetnyak's theorem \cites{reshetnyak1968} states that  $\tilde\alpha$ exists when $\alpha$ is closed.  If $\alpha$ is not closed, existence of $\tilde\alpha$ follows from applying Reshetnyak's theorem to the closed curve obtained by joining the endpoints of $\alpha$ with a geodesic segment, and then taking the corresponding subarc of the resulting majorizing curve. Here we obtain additional information about the curvature of $\tilde\alpha$:

\begin{theorem}\label{thm:reshetnyak-curvature}
Let 
$\alpha\colon[0,\ell]\to  X_k$ be a unit speed curve.  If $k>0$, assume
that the closed curve formed by joining the endpoints of $\alpha$ with a geodesic segment has length less than $2\pi/\sqrt{k}$ (e.g., $\ell<\pi/\sqrt{k}$). Then
there exists a majorization
$\tilde\alpha\colon[0,\ell]\to\M^2_k$ with
$
 \kappa_{\tilde\alpha}
 \leq
 \kappa_\alpha.
 $
 Furthermore, if $|\alpha(t)\alpha(s)|=|\tilde\alpha(t)\tilde\alpha(s)|$ for all $s,t\in[0,\ell]$, then the correspondence $\tilde\alpha(t)\mapsto\alpha(t)$ extends to a distance preserving map $\conv(\tilde\alpha)\to X_k$.
\end{theorem}

This result generalizes \cite[Prop. 4.3]{ghomi2026-total} where $\alpha$ was $\C^1$, $k=0$, and $ X_k$ was a Riemannian manifold. The majorizing curve  $\tilde\alpha$ has also been called an ``unfolding'' \cites{cks2002,ghomi-wenk2021} or ``chord-stretching'' \cites{sallee1973,brooks-strantzen1992}. 
We first consider the polygonal case of Theorem \ref{thm:reshetnyak-curvature}. 
A curve $\alpha\colon[0,\ell]\to X_k$ is \emph{polygonal} if there is
a partition $0=t_0<\cdots<t_N=\ell$ such that each subarc
$\alpha([t_{i-1},t_i])$ is a geodesic segment. The points $t_i$ are called \emph{vertices} of $\alpha$.

\begin{lemma}\label{lem:polygonal-majorization}
Let $\alpha\colon[0,\ell]\to  X_k$ be a polygonal curve, and
$\tilde\alpha\colon[0,\ell]\to\M^2_k$ be a majorization of
$\alpha$.  Then $\tilde\alpha$ is polygonal and
$
        \kappa_{\tilde\alpha}\leq \kappa_\alpha.
$
\end{lemma}

\begin{proof}
Let $p_i:=\alpha(t_i)$ be the vertices of $\alpha$, and set
$\tilde p_i:=\tilde\alpha(t_i)$. Since $\alpha$ is polygonal, each subarc
$\alpha|_{[t_{i-1},t_i]}$ is a geodesic segment. Thus
$$
|p_{i-1}p_i|=\length(\alpha|_{[t_{i-1},t_i]})=\length(\tilde\alpha|_{[t_{i-1},t_i]})\geq |\tilde p_{i-1}\tilde p_i|\geq |p_{i-1}p_i|.
$$
So $\length(\tilde\alpha|_{[t_{i-1},t_i]})= |\tilde p_{i-1}\tilde p_i|$, which yields that $\tilde\alpha|_{[t_{i-1},t_i]}$ is a geodesic segment. It remains to compare the angles. The adjacent sides of an interior vertex $p_i$ have the same length as those of $\tilde p_i$ while
$
        |p_{i-1}p_{i+1}|
        \leq
        |\tilde p_{i-1}\tilde p_{i+1}|.
$
Let $\overline p_{i-1}\overline p_i\overline p_{i+1}$ be a
triangle in $\M_k^2$ with the same side lengths as $p_{i-1}p_ip_{i+1}$.
 The triangle
$\tilde p_{i-1}\tilde p_i\tilde p_{i+1}$ has the same two
side lengths adjacent to $\tilde p_i$, and no shorter opposite side.  Thus the angles
$$
\angle(\tilde p_{i-1}\tilde p_i\tilde p_{i+1})
\geq
\angle(\overline p_{i-1}\overline p_i\overline p_{i+1})
\geq
\angle(p_{i-1} p_i p_{i+1}),
$$
where the second inequality holds since $X_k$ is a $\CAT(k)$ space.
Hence the exterior angles $\pi-\angle(\tilde p_{i-1}\tilde p_i\tilde p_{i+1})\leq \pi-\angle(p_{i-1} p_i p_{i+1})$, which completes the proof.
\end{proof}

Now we establish the main result of this section:

\begin{proof}[Proof of Theorem \ref{thm:reshetnyak-curvature}]
Let $A\subset[0,\ell]$ be a countable dense set containing $0$ and
$\ell$.  Choose a nested sequence of partitions
$$P_j=\left\{0=t_0^j<t_1^j<\cdots<t_{N_j}^j=\ell\right\},
\qquad P_j\subset P_{j+1},
\qquad |P_j|\to0,$$
so that every point of $A$ eventually belongs to $P_j$ for all sufficiently large
$j$.  Let $\alpha_j$ be the inscribed geodesic polygon with vertices
$\alpha(t_i^j)$, parametrized on $[0,\ell]$ by constant speed on each
interval $[t_{i-1}^j,t_i^j]$.  

Since $\alpha$ is unit speed, each $\alpha_j$ is $1$-Lipschitz. So $\alpha_j\to\alpha$
uniformly.  If $a,b\in A$, $a<b$, then for all sufficiently large $j$
both $a$ and $b$ belong to $P_j$, and
$\alpha_j|_{[a,b]}$ is the inscribed polygon of $\alpha|_{[a,b]}$
determined by  $P_j\cap[a,b]$, whose mesh tends to
zero.  By the definition of total curvature,
$$
\kappa_{\alpha_j}([a,b])\leq \kappa_\alpha([a,b]).
$$
Let
$\tilde\alpha_j\colon[0,\ell]\to\M^2_k$ be a majorization of
$\alpha_j$, parametrized so that the vertices corresponding to
$\alpha(t_i^j)$ occur at the same parameters $t_i^j$, and so that each edge is traversed with constant speed.  We may assume that these curves all have the same initial conditions.
By Lemma~\ref{lem:polygonal-majorization}, corresponding sides of
$\alpha_j$ and $\tilde\alpha_j$ have equal lengths. Hence the curves
$\tilde\alpha_j$ are $1$-Lipschitz,
and their images lie in a fixed compact ball of $\M^2_k$.  Thus, after
passing to a subsequence, they converge uniformly to a curve
$
        \tilde\alpha\colon[0,\ell]\to\M^2_k .
$

Since $\conv(\tilde\alpha_j)$ converge to $\conv(\tilde\alpha)$ in the
Hausdorff distance, the boundary curves $\partial\conv(\tilde\alpha_j)$,
with their arclength parametrizations, converge uniformly to 
$\partial\conv(\tilde\alpha)$, which may degenerate to a
doubly covered geodesic segment. Hence $\tilde\alpha$ is  a possibly degenerate
convex curve.  The majorization inequalities pass to
the limit:
$$|\tilde\alpha(s)\tilde\alpha(t)|
\geq
|\alpha(s)\alpha(t)|$$        
 for all $s,t\in[0,\ell]$.
Equality holds for $s=0$, $t=\ell$, because each $\tilde\alpha_j$ has endpoint distance equal to $|\alpha(0)\alpha(\ell)|$. It remains to check that
$\tilde\alpha$ has the same arclength function as $\alpha$.  Since
$\alpha$ is parametrized by arclength and each edge of $\alpha_j$ is
traversed with speed at most $1$, the same is true for
$\tilde\alpha_j$, because corresponding sides have equal lengths.  Hence
$\tilde\alpha$ is $1$-Lipschitz, and so
$$\operatorname{length}\bigl(\tilde\alpha|_{[a,b]}\bigr)
\leq b-a .$$
On the other hand, for every partition
$a=t_0<\cdots<t_N=b$, the majorization inequality gives
$
        \sum
        |\tilde\alpha(t_i)\tilde\alpha(t_{i+1})|
        \geq
        \sum
        |\alpha(t_i)\alpha(t_{i+1})|.
$
Taking the supremum over all partitions of $[a,b]$, and using that
$\alpha$ is parametrized by arclength, gives
$$\operatorname{length}\bigl(\tilde\alpha|_{[a,b]}\bigr)
\geq b-a .$$
Thus equality holds, and $\tilde\alpha$ has the same arclength function
as $\alpha$.  Therefore $\tilde\alpha$ is a majorization of
$\alpha$. The characterization of the equality case holds by the rigidity case of Reshetnyak's theorem \cite[9.61 \& p. 278]{akp2024}.

Finally, to obtain the curvature inequality, first let $a<b$ belong to
$A$.  By Lemma~\ref{lem:polygonal-majorization},
$\kappa_{\tilde\alpha_j}([a,b])\leq \kappa_{\alpha_j}([a,b])$.  Thus, by
lower semicontinuity of total curvature under uniform convergence
\cite[Thm. 5.1.1]{alexandrov-reshetnyak1989,karuwannapatana-maneesawarng2007},
$$\kappa_{\tilde\alpha}([a,b])
\leq
\liminf_{j\to\infty}\kappa_{\tilde\alpha_j}([a,b])
\leq
\limsup_{j\to\infty}\kappa_{\alpha_j}([a,b])
\leq
\kappa_\alpha([a,b]).$$
Hence the desired inequality holds for all intervals whose endpoints belong
to $A$.  Now let $[a,b]\subset[0,\ell]$ be arbitrary.  Since $A$ is dense, we may choose $a_m,b_m\in A$ with
$a<a_{m+1}<a_m<b_m<b_{m+1}<b$, $a_m\to a$, and $b_m\to b$.  Then
$$
\kappa_{\tilde\alpha}([a_m,b_m])\leq \kappa_\alpha([a_m,b_m])
$$
for every $m$.  As $m\to\infty$, the intervals $[a_m,b_m]$ increase to $(a,b)$.  Since, by definition,
total curvature involves only the interior points of 
partitions, 
$\kappa_{\tilde\alpha}([a_m,b_m])\to\kappa_{\tilde\alpha}([a,b])$ and
$\kappa_\alpha([a_m,b_m])\to\kappa_\alpha([a,b])$, completing the proof.
\end{proof}

\begin{note}
One might expect that the curvature condition $\kappa_{\tilde\alpha}\leq \kappa_{\alpha}$ should hold for all majorizing curves $\tilde\alpha$. This is the case  for $\C^{1,1}$ curves $\alpha$ in Cartan-Hadamard manifolds \cite[Lem. 4.7]{ghomi2026-total}. By Lemma \ref{lem:polygonal-majorization}, it suffices to show that any majorization $\tilde\alpha$ is the limit of majorizations of polygonal approximations of $\alpha$.
\end{note}

\section{The Main Result}\label{sec:main}
In this section we use the refinement of Reshetnyak's majorization (Theorem \ref{thm:reshetnyak-curvature}) together with the comparison in model planes (Proposition \ref{prop:main-model}) to establish the following result which sharpens Theorem \ref{thm:main}:

\begin{theorem}\label{thm:main2}
Let $\alpha$, $\beta$ be as in Theorem \ref{thm:main}, and $\tilde \beta$ be a majorization of $\beta$ furnished by Theorem \ref{thm:reshetnyak-curvature}, so that $\kappa_{\tilde\beta}\leq \kappa_{\beta}$. For $0\leq r\leq1$, let $\alpha_r\colon[0,\ell]\to\M^2_k$ be unit speed curves with curvature measure
$d\K_\alpha-r\,d(\K_\alpha-\K_{\tilde\beta}).$
 Then
\begin{equation*}
|\beta(0)\beta(\ell)|
-
|\alpha(0)\alpha(\ell)|
=
\int_0^1
\MA\bigl(\alpha_r,\K_\alpha-\K_{\tilde\beta}\bigr)dr .
\end{equation*}
In particular,
$
|\beta(0)\beta(\ell)|\geq |\alpha(0)\alpha(\ell)|.
$
Equality holds only if the correspondence
$\alpha(t)\mapsto\beta(t)$ extends to an isometry between $\conv(\alpha)$ and $\conv(\beta)$.
\end{theorem}

To establish the rigidity part of the above result, we need the following lemma
about angles in $\CAT(k)$ spaces, which is immediate in Riemannian manifolds. 
We review some basic facts, which may be found in \cite{bridson-haefliger1999,akp2024}.
For points
$p,o,q\in X_k$, with $p,q\neq o$, let $\Delta_k(poq)\subset \M^2_k$ be a geodesic triangle with side lengths $|po|$, $|qo|$, and $|pq|$. Let $\angle(poq)$ be the angle of $\Delta_k(poq)$ at $o$.  If $\gamma$, $\sigma\colon[0,\ell]\to X_k$ are unit speed geodesics with $\gamma(0)=o=\sigma(0)$, then the \emph{(Alexandrov) angle} between them at $o$ is defined as
$$
\angle(\gamma,\sigma):= \lim_{s,t\to0^+}\angle\big(\gamma(s)\,o\,\sigma(t)\big),
$$ 
\cite[6C, 9.14(c)]{akp2024}. This
defines a pseudometric on
geodesics starting at $o$, which determines the space of \emph{directions} $S_o(X_k)$, after passing to the quotient and metric completion \cite[6D]{akp2024} \cite[p. 190]{bridson-haefliger1999}.
If $\gamma$, $\sigma$ are unit speed curves of finite total curvature, then their right derivatives $\gamma'_+(0)$, $\sigma'_+(0)$ exist
\cite[Cor.~2.7]{maneesawarng-lenbury2003}, in the sense of
\cite[6.9, 6.12]{akp2024}, and define directions in $S_o(X_k)$. Then 
$\angle(\gamma,\sigma):=\angle(\gamma'_+(0), \sigma'_+(0))$, i.e., the angle between $\gamma$ and $\sigma$ is determined by the angles between approximating geodesics starting at $o$.

\begin{lemma}\label{lem:curve-geodesic-angle}
Let $\gamma\colon[0,a]\to X_k$,  $\sigma\colon[0,b]\to X_k$ be
unit speed curves with $\gamma(0)=o=\sigma(0)$. Suppose that $\gamma$
has finite total curvature, and $\sigma$ is a geodesic. If $k>0$, assume
that $b<\pi/\sqrt{k}$. Then
$$
\angle(\gamma,\sigma)=\lim_{t\to0^+}\angle(\gamma(t)\,o\,\sigma(b)).
$$
\end{lemma}

\begin{proof}
By the definition of
right derivative \cite[6.9]{akp2024}, there are unit speed geodesics $\gamma_i$
starting at $o$, with $(\gamma_i)'_+(0)\to \gamma'_+(0)$, such that
$$
\lim_{i\to\infty}\limsup_{t\to0^+}
\frac{|\gamma(t)\gamma_i(t)|}{t}=0.
$$
Since $\gamma_i$ is unit speed, $|o\gamma_i(t)|=t$. Hence
$
\left||o\gamma(t)|/t-1\right|
\leq
|\gamma(t)\gamma_i(t)|/t,
$
by the reverse triangle inequality,
and consequently
$
|o\gamma(t)|/t\to1.
$
Thus the side lengths of the comparison triangles
$\Delta_k(\gamma(t)\,o\,\sigma(b))$ and
$\Delta_k(\gamma_i(t)\,o\,\sigma(b))$ satisfy
$$
\big||o\gamma(t)|-|o\gamma_i(t)|\big|
\leq |\gamma(t)\gamma_i(t)|,\qquad
\big||\gamma(t)\sigma(b)|-|\gamma_i(t)\sigma(b)|\big|
\leq |\gamma(t)\gamma_i(t)|,
$$
while the side
$|o\sigma(b)|=b$ is fixed. Now  the law of cosines in $\M^2_k$ yields
$$
\lim_{i\to\infty}\limsup_{t\to0^+}
\big|
\angle(\gamma(t)\,o\,\sigma(b))
-\angle(\gamma_i(t)\,o\,\sigma(b))
\big|=0.
$$
By the ``strong angle lemma'' \cite[9.35]{akp2024},
$
\angle(\gamma_i,\sigma)=\lim_{t\to0^+}
\angle(\gamma_i(t)\,o\,\sigma(b)).
$
Furthermore, since $(\gamma_i)'_+(0)\to\gamma'_+(0)$, we have
$\angle(\gamma_i,\sigma)\to\angle(\gamma,\sigma)$. By the triangle inequality,
\begin{gather*}
\big|\angle(\gamma(t)\,o\,\sigma(b))-\angle(\gamma,\sigma)\big|
\leq\\
\big|\angle(\gamma(t)\,o\,\sigma(b))
-\angle(\gamma_i(t)\,o\,\sigma(b))\big| 
+
\big|\angle(\gamma_i(t)\,o\,\sigma(b))
-\angle(\gamma_i,\sigma)\big|
+
\big|\angle(\gamma_i,\sigma)-\angle(\gamma,\sigma)\big|.
\end{gather*}
Taking $\limsup$ as $t\to0^+$, and then
letting $i\to\infty$, proves the claim.
\end{proof}

Now we are ready to establish the main result of this work:

\begin{proof}[Proof of Theorem \ref{thm:main2}]
Since  $|\beta(0)\beta(\ell)|=|\tilde\beta(0)\tilde\beta(\ell)|$, the desired identity follows immediately from Theorem \ref{thm:moment-id}. 
Note that $\tilde\beta$ is not degenerate, and therefore is a convex
curve. Indeed, if $\tilde\beta$ double covered a geodesic segment then $d\K_{\tilde\beta}(\{t\})=\pi$ at some point $t\in (0,\ell)$. But $\kappa_{\tilde\beta}\leq\kappa_{\beta}\leq \kappa_\alpha$, which implies that $|d\K_{\tilde\beta}|\leq |d\K_\alpha|$, as noted in Section \ref{sec:curvature-measures}. Thus $|d\K_{\alpha}(\{t\})|=\pi$, which is impossible since convex curves cannot have cusps.
Now by Proposition~\ref{prop:main-model}, $|\tilde\beta(0)\tilde\beta(\ell)|\geq |\alpha(0)\alpha(\ell)|$. Thus we obtain the desired endpoint inequality
\begin{equation}\label{eq:alpha-beta}
|\beta(0)\beta(\ell)|\geq |\alpha(0)\alpha(\ell)|.
\end{equation}
It remains to characterize the equality case. Suppose that $|\beta(0)\beta(\ell)|=|\alpha(0)\alpha(\ell)|$. Then
$
|\tilde\beta(0)\tilde\beta(\ell)|
=
|\alpha(0)\alpha(\ell)|.
$
We also know that $\kappa_{\tilde\beta}\leq\kappa_\beta\leq\kappa_\alpha$. So Proposition~\ref{prop:main-model}, applied to $\alpha$ and
$\tilde\beta$, shows that the correspondence
$\alpha(t)\mapsto\tilde\beta(t)$ extends to an isometry
$G\colon \conv(\alpha)\to \conv(\tilde\beta)$. Thus
$\kappa_{\tilde\beta}=\kappa_\alpha$, and consequently
\begin{equation}\label{eq:kappas}
\kappa_{\tilde\beta}=\kappa_\beta.
\end{equation}
For
$0\leq s<t\leq\ell$, applying \eqref{eq:alpha-beta} to
the restrictions $\alpha|_{[s,t]}$ and $\beta|_{[s,t]}$, gives
$
|\beta(s)\beta(t)|\geq|\alpha(s)\alpha(t)|.
$
Since $G$ is an isometry,
$
|\alpha(s)\alpha(t)|=|\tilde\beta(s)\tilde\beta(t)|.
$
So 
$
|\beta(s)\beta(t)|\geq|\tilde\beta(s)\tilde\beta(t)|.
$
The reverse inequality holds by majorization. Hence
\begin{equation}\label{eq:betas}
|\beta(s)\beta(t)|
=
|\tilde\beta(s)\tilde\beta(t)|.
\end{equation}

Using the  observations above, we construct a distance preserving map
$F\colon\conv(\tilde\beta)\to X_k$ which sends $\tilde\beta(t)$ to
$\beta(t)$. Then $F\circ G$ is distance preserving and sends
$\alpha(t)$ to $\beta(t)$. Furthermore, $F\circ G(\conv(\alpha))=\conv(\beta)$.
Indeed, since $F\circ G$ is distance preserving, it sends geodesic segments
to geodesic segments. Thus $F\circ G(\conv(\alpha))$ is convex. It contains
$\beta$, and therefore $\conv(\beta)\subset F\circ G(\conv(\alpha))$.
Conversely, if $C\subset X_k$ is any convex set containing
$\beta$, then $(F\circ G)^{-1}(C)$ is a convex subset of
$\conv(\alpha)$ containing $\alpha$. Hence
$(F\circ G)^{-1}(C)=\conv(\alpha)$. Intersecting over all such $C$ gives
$F\circ G(\conv(\alpha))\subset\conv(\beta)$. So
$F\circ G(\conv(\alpha))=\conv(\beta)$, as desired. Thus it remains to construct $F$.

If $\alpha(0)=\alpha(\ell)$, then $\tilde\beta$ and $\beta$ are closed as
well. So the rigidity case of Theorem~\ref{thm:reshetnyak-curvature} yields
the desired map $F$. Next we consider the case
$\alpha(0)\neq\alpha(\ell)$.
Let $\sigma$ be the geodesic segment joining $\beta(0)$ to
$\beta(\ell)$. By construction, $\tilde\beta$  is the arc corresponding to
$\beta$ in a closed curve majorizing $\beta\cup\sigma$. Let $\tilde\sigma$
be the complementary arc. Then $\length(\tilde\sigma)=\length(\sigma)$. Furthermore, the endpoints of $\tilde\sigma$ are $\tilde\beta(0)$ and
$\tilde\beta(\ell)$, and $|\tilde\beta(0)\tilde\beta(\ell)|=|\beta(0)\beta(\ell)|=\length(\sigma)$. So
$\tilde\sigma$ is the geodesic segment joining $\tilde\beta(0)$ to
$\tilde\beta(\ell)$. Thus the correspondence
$$
f\colon\tilde\beta\cup\tilde\sigma\to \beta\cup\sigma,
$$
which sends $\tilde\beta$ to $\beta$ and $\tilde\sigma$ to $\sigma$, preserving arclength, is a majorization map.
We claim that $f$ preserves ambient distances. Then, by the rigidity
case of Theorem~\ref{thm:reshetnyak-curvature}, $f$ extends to a distance preserving map $F\colon\conv(\tilde\beta\cup\tilde\sigma)=\conv(\tilde\beta)\to X_k$, as desired.

So it remains to show that $|f(p)f(q)|=|pq|$ for all points $p$, $q\in \partial\conv(\tilde\beta)=\tilde\beta\cup\tilde\sigma$. 
This holds if $p$, $q\in\tilde\beta$ by \eqref{eq:betas}, and it is immediate if $p$, $q\in\tilde\sigma$. Finally, consider the case where
$p\in\tilde\beta$ and $q\in\tilde\sigma$. 
Choose an arc $\tilde\gamma$ in $\partial\conv(\tilde\beta)$ connecting $p$ and $q$ which realizes the intrinsic distance 
between these points in $\partial\conv(\tilde\beta)$. Then $\length(\tilde\gamma)<\pi/\sqrt{k}$ if $k>0$. For definiteness, assume that $\tilde\gamma$  passes
through $\tilde\beta(0)$ (the other case is identical). Then $\tilde\gamma$ consists of an arc
$\tilde\beta|_{[0,s]}$, with the reverse orientation, followed by the
subsegment of $\tilde\sigma$ from $\tilde\beta(0)$ to $q$. Let $\gamma$ be the
corresponding arc in $X_k$, formed by $\beta|_{[0,s]}$, with the reverse
orientation, followed by the corresponding subsegment of $\sigma$.

The curvature of $\tilde\gamma$  comes from
$\tilde\beta|_{[0,s]}$, and the exterior angle of $\tilde\gamma$ at $\tilde\beta(0)$, which is  $\pi-\angle(\tilde\beta,\tilde\sigma)$. Similarly, the curvature of
$\gamma$ consists of the curvature of $\beta|_{[0,s]}$ and the exterior angle $\pi-\angle(\beta,\sigma)$.  By \eqref{eq:kappas},
$\beta$ and $\tilde\beta$ have the same curvature. Furthermore, the triangles $\tilde\beta(t)\tilde\beta(0)\tilde\beta(\ell)$ and $\beta(t)\beta(0)\beta(\ell)$ have equal side lengths by \eqref{eq:betas}. Thus  Lemma~\ref{lem:curve-geodesic-angle} gives
$$
\angle(\beta,\sigma)
=
\lim_{t\to0^+}\angle\big(\beta(t)\beta(0)\beta(\ell)\big)
=
\lim_{t\to0^+}\angle\big(\tilde\beta(t)\tilde\beta(0)\tilde\beta(\ell)\big)
=
\angle(\tilde\beta,\tilde\sigma).
$$
Hence $\kappa_{\tilde\gamma}=\kappa_\gamma$. 
It follows that
$
|pq|\leq |f(p)f(q)|,
$
by the endpoint inequality \eqref{eq:alpha-beta} applied to $\tilde\gamma$ and
$\gamma$. The reverse inequality holds since $f$ is distance nonincreasing. Hence $f$ is distance preserving as claimed, which completes the proof.
\end{proof}

\appendix
\section{Polygonal Approximations}\label{sec:polygonal}

The chief difficulty in proving Theorem \ref{thm:main} was the characterization of the equality case in the endpoint inequality \eqref{eq:schur}, which necessitated the development of the moment arm identity \eqref{eq:moment-id}, and the variational technique used to  apply this identity in model planes. If, however, one is interested only in proving \eqref{eq:schur}, without rigidity, then it can be obtained more quickly using polygonal approximations which we develop here.
A convex curve $\alpha\colon[0,\ell]\to\M^2_k$ is called \emph{transversal} 
if it is not closed and it is not tangent at either endpoint to the geodesic $\L$ passing through its endpoints. 

\begin{proposition}\label{prop:polygonal-approx}
Let $\alpha,\beta\colon[0,\ell]\to\M^2_k$ be unit speed convex curves, with $\alpha$ transversal. Then there
are partitions $P_m=\{0=t^m_0<\cdots<t^m_{N_m}=\ell\}$, with
$|P_m|\to0$, and unit speed polygonal curves $\alpha_m,\beta_m\colon[0,\ell]\to\M^2_k$,
with vertices at $t^m_j$, such that $\alpha_m$ are convex,
$\alpha_m\to\alpha$ and $\beta_m\to\beta$ uniformly,  $\K_{\alpha_m}\to \K_\alpha$ and $\K_{\beta_m}\to\K_\beta$  in
$L^1([0,\ell])$, and
$
\K_{\alpha_m}(\ell)\to\K_\alpha(\ell)
$,
$
\K_{\beta_m}(\ell)\to\K_\beta(\ell).
$
If in addition $\kappa_\beta\leq\kappa_\alpha$, then we may also require that the exterior angles
$
\tilde\angle_{\beta_m}(t^m_j)\leq\tilde\angle_{\alpha_m}(t^m_j),
$
for $1\leq j\leq N_m-1$.
\end{proposition}

Now  \eqref{eq:schur} may be proved as follows. By Theorem \ref{thm:reshetnyak-curvature} it suffices to consider the case of convex curves $\alpha,\beta\colon[0,\ell]\to\M^2_k$. Let $\alpha_m$, $\beta_m$ be as in Proposition \ref{prop:polygonal-approx}.
By the  arm lemma \cite[9.63]{akp2024},  \eqref{eq:schur} holds for $\alpha_m$, $\beta_m$, and we let $m\to\infty$. In the general case, we may assume that $\alpha$ is not closed and is not a geodesic. Then
$\alpha|_{[0,\ell-\epsilon]}$ is transversal for small $\epsilon>0$, if no terminal segment $\alpha([b,\ell])$ lies on $\mathcal{L}$.  Hence \eqref{eq:schur} holds for $\alpha|_{[0,\ell-\epsilon]}$, $\beta|_{[0,\ell-\epsilon]}$, and we let $\epsilon\to 0$. If $\alpha([b,\ell])\subset\L$, we reverse the parametrization unless an initial segment $\alpha([0,a])\subset\L$. In that case \eqref{eq:schur} still holds for $\alpha|_{[a,b]}$ and $\beta|_{[a,b]}$ by the preceding argument, assuming $[0,a]\cup [b,\ell]=\alpha^{-1}(\L)$, and the triangle inequality yields:
$$
a+|\alpha(0)\alpha(\ell)|+(\ell-b)=|\alpha(a)\alpha(b)|\leq |\beta(a)\beta(b)|\leq a+|\beta(0)\beta(\ell)|+(\ell-b).
$$
Thus we obtain $|\alpha(0)\alpha(\ell)|\leq |\beta(0)\beta(\ell)|$ as desired.

The proof of Proposition \ref{prop:polygonal-approx} is divided into three parts. First we construct $\alpha_m$, $\beta_m$ and establish their convergence and local convexity. It then remains to
show that $\alpha_m$ are convex. For this we reduce the
problem to simplicity of the closed curves $\ol\alpha_m$ obtained by connecting the endpoints of $\alpha_m$ with a geodesic segment $\sigma_m$, using 
elementary estimates for the curvature of small subarcs. Finally we prove the simplicity of $\ol\alpha_m$ by ruling out self-intersections
of $\alpha_m$ and intersections of $\sigma_m$ with the interior of $\alpha_m$.

\subsection{Construction and convergence}\label{subsec:construction}
After possibly reflecting the model plane for $\alpha$ and $\beta$ separately,
we may assume that both curves turn left. Then $d\K_\alpha,d\K_\beta\geq0$,
so $\K_\alpha$ and $\K_\beta$ are monotone. Consequently $\K_\alpha$ and $\K_\beta$ have at most countably
many discontinuities, and we may choose the partitions $P_m$ so that their
interior points are continuity points of both $\K_\alpha$ and $\K_\beta$. Set
$I^m_j:=[t^m_{j-1},t^m_j]$ and $\Delta^m_j:=t^m_j-t^m_{j-1}$.
For $1\leq j\leq N_m-1$, define $a^m_j:=d\K_\alpha(I^m_j)$ and
$b^m_j:=d\K_\beta(I^m_j)$. Refining the partitions if necessary, we may
assume $a^m_j,b^m_j\in[0,\pi)$ for all $j$. 

Let $\alpha_m\colon[0,\ell]\to\M^2_k$ be the unit speed polygonal curve which starts at $\alpha(0)$, is tangent to
$\alpha$ at $\alpha(0)$, has successive edge lengths $\Delta^m_j$, and turns left by $a^m_j$ at $t^m_j$, i.e.,
$d\K_{\alpha_m}(\{t^m_j\})=a^m_j$. Similarly, let $\beta_m\colon[0,\ell]\to\M^2_k$ be the polygonal curve
which starts at $\beta(0)$, is tangent to $\beta$ at $\beta(0)$, has successive
edge lengths $\Delta^m_j$, and turns left by $b^m_j$ at $t^m_j$, i.e.,
$d\K_{\beta_m}(\{t^m_j\})=b^m_j$. In particular, $\alpha_m$ and $\beta_m$  are locally convex. If $\kappa_\beta\leq\kappa_\alpha$, we have
$d\K_\beta\leq d\K_\alpha$, which yields $b^m_j\leq a^m_j$. So 
$\tilde\angle_{\beta_m}(t^m_j)\leq\tilde\angle_{\alpha_m}(t^m_j).$

Let $s\in(0,\ell)$ be a continuity point of $\K_\alpha$, and let $t^m_{j_m}$
be the largest point of $P_m$ with $t^m_{j_m}<s$. Then $t^m_{j_m}\to s$, and
by construction $\K_{\alpha_m}(s)=\K_\alpha(t^m_{j_m})$. Hence
$\K_{\alpha_m}(s)\to\K_\alpha(s)$. Similarly
$\K_{\beta_m}(s)\to\K_\beta(s)$ at every continuity point of $\K_\beta$.
Since these functions are monotone and uniformly bounded, it follows that
$$
\|\K_{\alpha_m}-\K_\alpha\|_{L^1([0,\ell])}\to0,
\qquad
\|\K_{\beta_m}-\K_\beta\|_{L^1([0,\ell])}\to0.
$$
By Lemma~\ref{lem:continuous}, $\alpha_m\to\alpha$ and
$\beta_m\to\beta$ uniformly on $[0,\ell]$.
Since the interior points of $P_m$ are continuity points of $\K_\alpha$,
splitting $\alpha$ at these points introduces no additional exterior angle.
Thus
$
\K_{\alpha_m}(\ell)=\K_\alpha(t^m_{N_m-1})\to\K_\alpha(\ell),
$
by the left-continuity of $\K_\alpha$. Similarly,
$
\K_{\beta_m}(\ell)\to\K_\beta(\ell).
$
Consequently, Lemma~\ref{lem:continuous} gives
$
T_{\alpha_m}(\ell)\to T_\alpha(\ell).
$

\subsection{Reduction to simplicity}
It remains to show that $\alpha_m$ is convex (for large $m$). Let $\ol\alpha_m$, $\ol\alpha$ be the closed curves formed by joining the endpoints of $\alpha_m$, $\alpha$, respectively, with  geodesic segments $\sigma_m$ and $\sigma$. It suffices to check that $\ol\alpha_m$ is simple. Indeed if $\ol\alpha_m$ is simple, then we can consider its interior angles. By construction, these angles are $\leq\pi$ at interior vertices of $\alpha_m$. Furthermore, the interior angles of $\ol\alpha_m$ at $\alpha_m(0)$ and $\alpha_m(\ell)$ are $\leq\pi$ as well, since by Lemma \ref{lem:continuous} these angles converge to the corresponding angles of $\ol\alpha$, which are $<\pi$ by the transversality assumption on $\alpha$. Hence $\ol\alpha_m$ is locally convex. Since $\ol\alpha_m$ is simple, it follows that it is convex, as desired.
To show that $\ol\alpha_m$ is simple, we need the following estimates. 

\begin{lemma}\label{lem:shrinking-curvature}
Let $r_m\leq u_m$ be sequences of points in $[0,\ell]$. If $r_m,u_m\to c$ for $c\in (0,\ell)$, then 
$$
\limsup_{m\to\infty}\kappa_\alpha([r_m,u_m])\leq \tilde\angle_\alpha(c).
$$
If $u_m\to0$ or $r_m\to\ell$, then
$
\kappa_\alpha([0,u_m])\to0
$
or
$
\kappa_\alpha([r_m,\ell])\to0,
$
respectively.
\end{lemma}

\begin{proof}
Let $\mu:=d\K_\alpha$. Since $\alpha$ turns left, $\mu\geq0$, and therefore,
as discussed in Section~\ref{sec:curvature-measures},
$
\kappa_\alpha([s,t])=\mu((s,t))
$
and
$
\mu(\{c\})=\tilde\angle_\alpha(c).
$
Since $\mu$ is a finite measure, it is continuous from above. In particular,  we may choose $r<c<u$ so close to $c$ that
$
\mu((r,u))<\mu(\{c\})+\epsilon,
$
for any given $\epsilon>0$.
For sufficiently large $m$, $[r_m,u_m]\subset (r,u)$. Hence
$$
\kappa_\alpha([r_m,u_m])
=
\mu((r_m,u_m))
\leq
\mu((r,u))
<
\tilde\angle_\alpha(c)+\epsilon,
$$
which proves the first assertion.  If $u_m\to0$, then after passing to a subsequence we may assume that $u_m$ is monotone decreasing. Hence, by continuity from above,
$
\mu((0,u_m))\to\mu(\emptyset)=0.
$
Similarly, if $r_m\to\ell$, then $\mu((r_m,\ell))
\to0.$
\end{proof}

\begin{lemma}\label{lem:construction-curvature-bound}
Let $[s,t]\subset[0,\ell]$. Let $r$ be the largest point of $P_m$ with
$r\leq s$, and $u$ be the smallest point of $P_m$ with $t\leq u$. Then
$\kappa_{\alpha_m}([s,t])\leq\kappa_\alpha([r,u])$.
\end{lemma}

\begin{proof}
By definition $\kappa_{\alpha_m}([s,t])$ is the sum of the exterior angles of
$\alpha_m$ at the vertices $t^m_j\in(s,t)$. By construction, the exterior
angle at $t^m_j$ is $a^m_j=d\K_\alpha(I^m_j)$, where
$I^m_j=[t^m_{j-1},t^m_j]$. Each such interval $I^m_j$ is contained in
$[r,u]$. Since the points of $P_m$ in $(0,\ell)$ are continuity points of
$\K_\alpha$, splitting $\alpha$ at these points introduces no extra curvature.
Thus the sum of the corresponding exterior angles is bounded above by
$\kappa_\alpha([r,u])$.
\end{proof}

\subsection{Proof of simplicity}

To establish that 
$\ol\alpha_m$ is simple, we first show that $\alpha_m$ is
simple, and then we check that $\alpha_m$ intersects  $\sigma_m$
only at its endpoints.

\subsubsection{Simplicity of $\alpha_m$}
Suppose, towards a contradiction, that  $\alpha_m$ is not simple. Passing to a subsequence, we
may choose $s_m<t_m$ such that
$
\alpha_m(s_m)=\alpha_m(t_m),
$
and $\alpha_m([s_m,t_m])$ is simple except for its endpoints.
Since $\alpha_m\to\alpha$ uniformly and $\alpha$ is injective, 
$s_m,t_m\to c$ for some $c\in[0,\ell]$. Thus the loop $\alpha_m([s_m,t_m])$ shrinks to a point. By the same
Gauss-Bonnet argument as in the proof of Lemma~\ref{lem:local-convexity},
since the exterior angle at the base point is at most $\pi$,
$$
\liminf_{m\to\infty}\kappa_{\alpha_m}([s_m,t_m])\geq \pi.
$$
On the other hand, let $r_m,u_m\in P_m$ be the points given by Lemma~\ref{lem:construction-curvature-bound}
for the interval $[s_m,t_m]$.
Then $r_m,u_m\to c$, and Lemma~\ref{lem:construction-curvature-bound} gives
$\kappa_{\alpha_m}([s_m,t_m])\leq \kappa_\alpha([r_m,u_m])$.
If $c\in(0,\ell)$, Lemma~\ref{lem:shrinking-curvature} yields
$$
\limsup_{m\to\infty}\kappa_{\alpha_m}([s_m,t_m])
\leq \tilde\angle_\alpha(c)<\pi.
$$
 If $c=0$ or
$c=\ell$, Lemmas \ref{lem:shrinking-curvature} and \ref{lem:construction-curvature-bound} give instead
$
\kappa_{\alpha_m}([s_m,t_m])\to0.
$
Both alternatives contradict the lower bound above.

\subsubsection{Intersection of $\alpha_m$ with $\sigma_m$} 
It remains to show that $\alpha_m((0,\ell))\cap\sigma_m=\emptyset$.
Suppose towards a contradiction that there exist
$s_m\in(0,\ell)$ with $\alpha_m(s_m)\in\sigma_m$. Since
$\alpha_m\to\alpha$ uniformly and $\sigma_m\to\sigma$, any limit point $c$ of
$s_m$ satisfies $\alpha(c)\in\sigma$. Since $\alpha$ is a transversal convex
arc, $c=0$ or $c=\ell$. If $c=0$, choose $s_m$ to be the first intersection of
$\alpha_m$ with $\sigma_m$ after $\alpha_m(0)$; if $c=\ell$, choose $s_m$ to
be the last such intersection before $\alpha_m(\ell)$. Then the corresponding
subarc of $\alpha_m$, together with the appropriate subsegment of $\sigma_m$,
forms a simple loop, and either $s_m\to 0$ or $s_m\to\ell$.

Suppose first that $s_m\to0$. Let $\psi_m$ be the angle between
$T_{\alpha_m}(0)$ and the initial direction of $\sigma_m$ at $\alpha_m(0)$. Since $\alpha_m$ has the
same initial tangent as $\alpha$, and since $\sigma_m\to\sigma$, we have
$
\psi_m\to\psi>0,
$
where $\psi$ is the angle between $T(0)$ and the initial direction of
$\sigma$ at $\alpha(0)$. The loop formed by $\alpha_m([0,s_m])$ and the geodesic segment from
$\alpha_m(s_m)$ back to $\alpha_m(0)$ has diameter tending to zero. As in the
Gauss-Bonnet argument in the proof of Lemma~\ref{lem:local-convexity}, the
exterior angle at $\alpha_m(0)$ is $\pi-\psi_m$, while the other exterior
angle is at most $\pi$, and therefore
$
\kappa_{\alpha_m}([0,s_m])\geq \psi_m-o(1).
$
Consequently,
$$
\liminf_{m\to\infty}\kappa_{\alpha_m}([0,s_m])\geq \psi>0.
$$
On the other hand, choose $u_m\in P_m$ with $s_m\leq u_m$ and
$u_m-s_m\leq |P_m|$. Then $u_m\to0$, and
Lemma~\ref{lem:construction-curvature-bound} gives
$
\kappa_{\alpha_m}([0,s_m])
\leq
\kappa_\alpha([0,u_m]).
$
This tends to $0$ by Lemma~\ref{lem:shrinking-curvature}, a contradiction. 

The case $s_m\to\ell$ is analogous. The only difference is that the angle at
$\alpha_m(\ell)$ is controlled by the convergence $T_{\alpha_m}(\ell)\to T(\ell)$,
obtained above from Lemma~\ref{lem:continuous}, instead
of the fixed initial tangent. Thus that angle is bounded away from zero, while
Lemmas~\ref{lem:construction-curvature-bound} and
\ref{lem:shrinking-curvature} force
$\kappa_{\alpha_m}([s_m,\ell])\to0$, a contradiction.

Thus $\alpha_m((0,\ell))\cap\sigma_m=\emptyset$, and so $\ol\alpha_m$ is
simple. As explained above, this implies that $\alpha_m$ is convex and
completes the proof of Proposition~\ref{prop:polygonal-approx}.

\section*{Acknowledgement}
We thank Anton Petrunin for useful discussions and encouragement to obtain a comprehensive form of Schur's comparison theorem.

\bibliography{references}

\begin{bibdiv}
\begin{biblist}

\bib{aigner-ziegler1999}{book}{
      author={Aigner, Martin},
      author={Ziegler, G\"{u}nter~M.},
       title={Proofs from {T}he {B}ook},
   publisher={Springer-Verlag, Berlin},
        date={1999},
        ISBN={3-540-63698-6},
        note={Including illustrations by Karl H. Hofmann, Corrected reprint of
  the 1998 original},
      review={\MR{1723092}},
}

\bib{akp2024}{book}{
      author={Alexander, Stephanie},
      author={Kapovitch, Vitali},
      author={Petrunin, Anton},
       title={Alexandrov geometry---foundations},
      series={Graduate Studies in Mathematics},
   publisher={American Mathematical Society, Providence, RI},
        date={[2024] \copyright 2024},
      volume={236},
        ISBN={[9781470473020]; [9781470475369]; [9781470475352]},
      review={\MR{4734965}},
}

\bib{alexander-bishop1996}{article}{
      author={Alexander, Stephanie~B.},
      author={Bishop, Richard~L.},
       title={Comparison theorems for curves of bounded geodesic curvature in
  metric spaces of curvature bounded above},
        date={1996},
        ISSN={0926-2245},
     journal={Differential Geom. Appl.},
      volume={6},
      number={1},
       pages={67\ndash 86},
         url={https://doi.org/10.1016/0926-2245(96)00008-3},
      review={\MR{1384880}},
}

\bib{alexander-bishop1998}{article}{
      author={Alexander, Stephanie~B.},
      author={Bishop, Richard~L.},
       title={The {F}ary-{M}ilnor theorem in {H}adamard manifolds},
        date={1998},
        ISSN={0002-9939,1088-6826},
     journal={Proc. Amer. Math. Soc.},
      volume={126},
      number={11},
       pages={3427\ndash 3436},
         url={https://doi.org/10.1090/S0002-9939-98-04423-2},
      review={\MR{1459103}},
}

\bib{alexandrov-reshetnyak1989}{book}{
      author={Alexandrov, A.~D.},
      author={Reshetnyak, Yu.~G.},
       title={General theory of irregular curves},
      series={Mathematics and its Applications (Soviet Series)},
   publisher={Kluwer Academic Publishers Group, Dordrecht},
        date={1989},
      volume={29},
        ISBN={90-277-2811-9},
         url={https://doi.org/10.1007/978-94-009-2591-5},
        note={Translated from the Russian by L. Ya. Yuzina},
      review={\MR{1117220}},
}

\bib{bridson-haefliger1999}{book}{
      author={Bridson, Martin~R.},
      author={Haefliger, Andr\'{e}},
       title={Metric spaces of non-positive curvature},
      series={Grundlehren der mathematischen Wissenschaften [Fundamental
  Principles of Mathematical Sciences]},
   publisher={Springer-Verlag, Berlin},
        date={1999},
      volume={319},
        ISBN={3-540-64324-9},
         url={https://doi.org/10.1007/978-3-662-12494-9},
      review={\MR{1744486}},
}

\bib{bbi2001}{book}{
      author={Burago, Dmitri},
      author={Burago, Yuri},
      author={Ivanov, Sergei},
       title={A course in metric geometry},
      series={Graduate Studies in Mathematics},
   publisher={American Mathematical Society, Providence, RI},
        date={2001},
      volume={33},
        ISBN={0-8218-2129-6},
      review={\MR{1835418 (2002e:53053)}},
}

\bib{cks2002}{article}{
      author={Cantarella, Jason},
      author={Kusner, Robert~B.},
      author={Sullivan, John~M.},
       title={On the minimum ropelength of knots and links},
        date={2002},
        ISSN={0020-9910},
     journal={Invent. Math.},
      volume={150},
      number={2},
       pages={257\ndash 286},
         url={https://doi.org/10.1007/s00222-002-0234-y},
      review={\MR{1933586}},
}

\bib{lopez-mateos-masque2010}{article}{
      author={Castrill\'{o}n~L\'{o}pez, M.},
      author={Fern\'{a}ndez~Mateos, V.},
      author={Mu\~{n}oz Masqu\'{e}, J.},
       title={Total curvature of curves in {R}iemannian manifolds},
        date={2010},
        ISSN={0926-2245,1872-6984},
     journal={Differential Geom. Appl.},
      volume={28},
      number={2},
       pages={140\ndash 147},
         url={https://doi.org/10.1016/j.difgeo.2009.10.008},
      review={\MR{2594458}},
}

\bib{chern1967}{incollection}{
      author={Chern, S.~S.},
       title={Curves and surfaces in {E}uclidean space},
        date={1967},
   booktitle={Studies in {G}lobal {G}eometry and {A}nalysis},
   publisher={Math. Assoc. America},
       pages={16\ndash 56},
      review={\MR{212744}},
}

\bib{ghomi2025-convexity}{article}{
      author={Ghomi, Mohammad},
       title={Convexity and rigidity of hypersurfaces in {C}artan-{H}adamard
  manifolds},
        date={2025},
     journal={Asian J. Math.},
      volume={29},
      number={6},
       pages={761\ndash 772},
}

\bib{ghomi2026-total}{article}{
      author={Ghomi, Mohammad},
      author={Hoisington, Joseph~Ansel},
      author={Raffaelli, Matteo},
      author={Stavroulakis, John~Ioannis},
       title={Total absolute curvature and rigidity of surfaces in
  {C}artan-{H}adamard manifolds},
        date={2026},
     journal={arXiv:2604.25024},
        note={\url{https://arxiv.org/abs/2604.25024}},
}

\bib{ghomi-wenk2021}{article}{
      author={Ghomi, Mohammad},
      author={Wenk, James},
       title={Shortest closed curve to inspect a sphere},
        date={2021},
        ISSN={0075-4102},
     journal={J. Reine Angew. Math.},
      volume={781},
       pages={57\ndash 84},
         url={https://doi.org/10.1515/crelle-2021-0049},
      review={\MR{4343095}},
}

\bib{gromov2025}{article}{
      author={Gromov, Misha},
       title={Lectures on immersions with controlled curvatures},
        date={2025},
     journal={arXiv:2511.01796},
        note={\url{https://arxiv.org/abs/2511.01796}},
}

\bib{hopf1989}{book}{
      author={Hopf, Heinz},
       title={Differential geometry in the large},
     edition={Second},
   publisher={Springer-Verlag},
     address={Berlin},
        date={1989},
        ISBN={3-540-51497-X},
        note={Notes taken by Peter Lax and John W. Gray, With a preface by S.
  S. Chern, With a preface by K. Voss},
      review={\MR{90f:53001}},
}

\bib{howard2024}{article}{
      author={Howard, Ralph},
       title={Schur type comparison theorems for affine curves},
        date={2024},
        ISSN={0047-2468,1420-8997},
     journal={J. Geom.},
      volume={115},
      number={1},
       pages={Paper No. 3, 23},
         url={https://doi.org/10.1007/s00022-023-00700-7},
      review={\MR{4679441}},
}

\bib{karuwannapatana-maneesawarng2007}{article}{
      author={Karuwannapatana, Wichitra},
      author={Maneesawarng, Chaiwat},
       title={The lower semi-continuity of total curvature in spaces of
  curvature bounded above},
        date={2007},
        ISSN={1513-489X},
     journal={East-West J. Math.},
      volume={9},
      number={1},
       pages={1\ndash 9},
      review={\MR{2444119}},
}

\bib{lang-schroeder1997}{article}{
      author={Lang, U.},
      author={Schroeder, V.},
       title={Kirszbraun's theorem and metric spaces of bounded curvature},
        date={1997},
        ISSN={1016-443X},
     journal={Geom. Funct. Anal.},
      volume={7},
      number={3},
       pages={535\ndash 560},
         url={https://doi.org/10.1007/s000390050018},
      review={\MR{1466337}},
}

\bib{lopez2011}{article}{
      author={L\'{o}pez, Rafael},
       title={The theorem of {S}chur in the {M}inkowski plane},
        date={2011},
        ISSN={0393-0440,1879-1662},
     journal={J. Geom. Phys.},
      volume={61},
      number={1},
       pages={342\ndash 346},
         url={https://doi.org/10.1016/j.geomphys.2010.10.009},
      review={\MR{2747005}},
}

\bib{maneesawarng-lenbury2003}{article}{
      author={Maneesawarng, Chaiwat},
      author={Lenbury, Yongwimon},
       title={Total curvature and length estimate for curves in {${\rm
  CAT}(K)$} spaces},
        date={2003},
        ISSN={0926-2245,1872-6984},
     journal={Differential Geom. Appl.},
      volume={19},
      number={2},
       pages={211\ndash 222},
         url={https://doi.org/10.1016/S0926-2245(03)00031-7},
      review={\MR{2002660}},
}

\bib{mendes2025}{article}{
      author={Mendes, Ricardo A.~E.},
       title={Diameter and focal radius of submanifolds},
        date={2025},
        ISSN={0232-704X,1572-9060},
     journal={Ann. Global Anal. Geom.},
      volume={68},
      number={2},
       pages={Paper No. 4, 6},
         url={https://doi.org/10.1007/s10455-025-10010-7},
      review={\MR{4927774}},
}

\bib{milnor1950}{article}{
      author={Milnor, J.~W.},
       title={On the total curvature of knots},
        date={1950},
     journal={Ann. of Math. (2)},
      volume={52},
       pages={248\ndash 257},
      review={\MR{12,273c}},
}

\bib{mucci-saracco2021}{article}{
      author={Mucci, Domenico},
      author={Saracco, Alberto},
       title={The total intrinsic curvature of curves in {R}iemannian
  surfaces},
        date={2021},
        ISSN={0009-725X,1973-4409},
     journal={Rend. Circ. Mat. Palermo (2)},
      volume={70},
      number={1},
       pages={521\ndash 557},
         url={https://doi.org/10.1007/s12215-020-00516-3},
      review={\MR{4234326}},
}

\bib{mucci-saracco2024}{article}{
      author={Mucci, Domenico},
      author={Saracco, Alberto},
       title={On the curvatures of irregular curves in {E}uclidean spaces and
  {R}iemannian surfaces},
        date={2024},
        ISSN={2189-3756,2189-3764},
     journal={Pure Appl. Funct. Anal.},
      volume={9},
      number={2},
       pages={555\ndash 600},
      review={\MR{4755827}},
}

\bib{ni2023}{article}{
      author={Ni, Lei},
       title={A {S}chur's theorem via a monotonicity and the expansion module},
        date={2023},
        ISSN={0075-4102,1435-5345},
     journal={J. Reine Angew. Math.},
      volume={805},
       pages={101\ndash 114},
         url={https://doi.org/10.1515/crelle-2023-0061},
      review={\MR{4669033}},
}

\bib{petrunin2025}{article}{
      author={Petrunin, Anton},
       title={Veronese minimizes normal curvatures},
        date={2025},
     journal={arXiv:2408.05909},
        note={\url{https://arxiv.org/abs/2408.05909}},
}

\bib{petrunin-zamora2022}{article}{
      author={Petrunin, Anton},
      author={Barrera, Sergio~Zamora},
       title={What is differential geometry: curves and surfaces},
        date={2022},
     journal={arXiv:2012.11814},
        note={\url{https://arxiv.org/pdf/2012.11814}},
}

\bib{reshetnyak1968}{article}{
      author={Reshetnyak, Yu.~G.},
       title={Non-expansive maps in a space of curvature no greater than
  {$K$}},
        date={1968},
        ISSN={0037-4474},
     journal={Sibirsk. Mat. \v{Z}.},
      volume={9},
       pages={918\ndash 927 [English translation in Siberian Math. J.
  \textbf{9} (1968), 683\ndash 687]},
      review={\MR{0244922}},
}

\bib{reshetnyak2005}{article}{
      author={Reshetnyak, Yu.~G.},
       title={The theory of curves in differential geometry from the point of
  view of the theory of functions of a real variable},
        date={2005},
        ISSN={0042-1316,2305-2872},
     journal={Uspekhi Mat. Nauk},
      volume={60},
      number={6(366)},
       pages={157\ndash 174},
         url={https://doi.org/10.1070/RM2005v060n06ABEH004286},
      review={\MR{2215759}},
}

\bib{rudin1976}{book}{
      author={Rudin, Walter},
       title={Principles of mathematical analysis},
     edition={Third},
      series={International Series in Pure and Applied Mathematics},
   publisher={McGraw-Hill Book Co., New York-Auckland-D\"{u}sseldorf},
        date={1976},
      review={\MR{385023}},
}

\bib{sabitov2004}{article}{
      author={Sabitov, I.~Kh.},
       title={Around the proof of the {L}egendre-{C}auchy lemma on convex
  polygons},
        date={2004},
        ISSN={0037-4474},
     journal={Sibirsk. Mat. Zh.},
      volume={45},
      number={4},
       pages={892\ndash 919},
  url={https://doi-org.prx.library.gatech.edu/10.1023/B:SIMJ.0000035837.80962.0a},
      review={\MR{2091654}},
}

\bib{sallee1973}{article}{
      author={Sallee, G.~T.},
       title={Stretching chords of space curves},
        date={1973},
     journal={Geometriae Dedicata},
      volume={2},
       pages={311\ndash 315},
         url={https://doi.org/10.1007/BF00181475},
      review={\MR{336560}},
}

\bib{schoenberg-zaremba1967}{article}{
      author={Schoenberg, I.~J.},
      author={Zaremba, S.~C.},
       title={On {C}auchy's lemma concerning convex polygons},
        date={1967},
        ISSN={0008-414X,1496-4279},
     journal={Canadian J. Math.},
      volume={19},
       pages={1062\ndash 1071},
         url={https://doi.org/10.4153/CJM-1967-096-4},
      review={\MR{216365}},
}

\bib{schur1921}{article}{
      author={Schur, Axel},
       title={{\"U}ber die schwarzsche extremaleigenschaft des kreises unter
  den kurven konstanter kr{\"u}mmung},
        date={1921},
     journal={Mathematische Annalen},
      volume={83},
       pages={143\ndash 148},
}

\bib{brooks-strantzen1992}{article}{
      author={Strantzen, John},
      author={Brooks, Jeff},
       title={A chord-stretching map of a convex loop is an isometry},
        date={1992},
        ISSN={0046-5755},
     journal={Geom. Dedicata},
      volume={41},
      number={1},
       pages={51\ndash 62},
         url={https://doi.org/10.1007/BF00181542},
      review={\MR{1147501}},
}

\bib{sullivan2008}{incollection}{
      author={Sullivan, John~M.},
       title={Curves of finite total curvature},
        date={2008},
   booktitle={Discrete differential geometry},
      series={Oberwolfach Semin.},
      volume={38},
   publisher={Birkh\"auser, Basel},
       pages={137\ndash 161},
         url={http://dx.doi.org/10.1007/978-3-7643-8621-4_7},
      review={\MR{2405664}},
}

\end{biblist}
\end{bibdiv}

\end{document}